\documentclass[english,3p]{elsarticle}
\usepackage[latin9]{inputenc}
\usepackage{mathtools}
\usepackage{amsmath}
\usepackage{amsthm}
\usepackage{amssymb}

\usepackage{xcolor}
\usepackage[colorlinks,
linkcolor=red,
anchorcolor=yellow,
citecolor=blue,
]{hyperref}

\makeatletter

\DeclareTextSymbolDefault{\textquotedbl}{T1}
\newcommand{\lyxdot}{.}

\theoremstyle{plain}
\ifx\thesection\undefined
\newtheorem{thm}{\protect\theoremname}
\else
\newtheorem{thm}{\protect\theoremname}[section]
\fi
\theoremstyle{remark}
\ifx\thesection\undefined
\newtheorem{rem}{\protect\remarkname}
\else
\newtheorem{rem}{\protect\remarkname}[section]
\fi
\theoremstyle{corollary}
\ifx\thesection\undefined
\newtheorem{cor}{\protect\corollaryname}
\else
\newtheorem{cor}{\protect\corollaryname}[section]
\fi
\theoremstyle{plain}
\ifx\thesection\undefined
\newtheorem{lem}{\protect\lemmaname}
\else
\newtheorem{lem}{\protect\lemmaname}[section]
\fi
\theoremstyle{definition}
\ifx\thesection\undefined
\newtheorem{example}{\protect\examplename}
\else
\newtheorem{example}{\protect\examplename}[section]
\fi

\allowdisplaybreaks

\@ifundefined{showcaptionsetup}{}{%
 \PassOptionsToPackage{caption=false}{subfig}}
\usepackage{subfig}
\makeatother

\providecommand{\corollaryname}{Corollary}
\providecommand{\examplename}{Example}
\providecommand{\lemmaname}{Lemma}
\providecommand{\remarkname}{Remark}
\providecommand{\theoremname}{Theorem}

\begin{document}

\title{Unconditional stability and error analysis of an Euler IMEX-SAV scheme
for the micropolar Navier-Stokes equations}

\author[zzu,zzu1]{Xiaodi Zhang}

\ead{zhangxiaodi@lsec.cc.ac.cn}

\author[huel]{Xiaonian Long\corref{cor1}}
\ead{longxiaonian@lsec.cc.ac.cn}

\cortext[cor1]{Corresponding author}

\address[zzu]{Henan Academy of Big Data, Zhengzhou University, Zhengzhou 450052,
	China.}

\address[zzu1]{School of Mathematics and Statistics, Zhengzhou University, Zhengzhou
	450001, China.}

\address[huel]{College of Mathematics and Information Science, Henan University of Economics and Law, Zhengzhou, 450047, China.}

\begin{frontmatter}{}
\begin{abstract}
In this paper, we consider numerical approximations for solving the
micropolar Navier-Stokes (MNS) equations, that couples the Navier-Stokes
equations and the angular momentum equation together. By combining
the scalar auxiliary variable (SAV) approach for the convective terms
and some subtle implicit-explicit (IMEX) treatments for the coupling
terms, we propose a decoupled, linear and unconditionally energy stable
scheme for this system. We further derive rigorous error estimates
for the velocity, pressure and angular velocity in two dimensions
without any condition on the time step. Numerical examples are
presented to verify the theoretical findings and show the performances of the scheme.
\end{abstract}

\end{frontmatter}{}

\section{Introduction}

Let $\Omega$ be a convex polygonal/polyhedral with boundary
$\Gamma\coloneqq\partial\Omega$ in $\mathbb{R}^{d},\,d=2,3$.
In this paper, we consider numerical approximation of the following
MNS equations:
\begin{subequations}
\begin{align}
\boldsymbol{u}_{t}+\boldsymbol{u}\cdot\nabla\boldsymbol{u}-(\nu+\nu_{r})\Delta\boldsymbol{u}+\nabla p-2\nu_{r}\nabla\times\boldsymbol{w}=\boldsymbol{0}\quad & \text{in}\quad\Omega\times J,\label{eq:modelu}\\
\nabla\cdot\boldsymbol{u}=0\quad & \text{in}\quad\Omega\times J,\label{eq:modeldivu}\\
\jmath\boldsymbol{w}_{t}+\jmath\boldsymbol{u}\cdot\nabla\boldsymbol{w}-\left(c_{a}+c_{d}\right)\Delta\boldsymbol{w}-\left(c_{0}+c_{d}-c_{a}\right)\nabla\nabla\cdot\boldsymbol{w}+4\nu_{r}\boldsymbol{w}-2\nu_{r}\nabla\times\boldsymbol{u}=\boldsymbol{0}\quad & \text{in}\quad\Omega\times J,\label{eq:modelw}
\end{align}
\label{eq:MNS_model} 
\end{subequations}
with boundary and initial conditions 
\begin{align*}
\boldsymbol{u}=\boldsymbol{0},\ \boldsymbol{w}=\boldsymbol{0}\quad{\rm on}\quad\Gamma\times J,\\
\boldsymbol{u}(\boldsymbol{x},0)=\boldsymbol{u}^{0}(\boldsymbol{x}),\quad\boldsymbol{w}(\boldsymbol{x},0)=\boldsymbol{w}^{0}(\boldsymbol{x})\quad{\rm in}\quad\Omega,
\end{align*}
where $T>0$ is the final time, $J=(0,T]$, $(\boldsymbol{u},p,\boldsymbol{w})$
represent the the linear velocity, pressure and angular velocity.
All the material constants $\jmath,\nu,\nu_{r},c_{a},c_{d}$ and $c_{0}$
are the kinematic viscosity which are assumed to be constant, positive
and satisfy $c_{0}+c_{d}-c_{a}>0$. Moreover, $\nu$ is the usual
Newtonian viscosity, and $\nu_{r}$ is the microrotation viscosity.
In order to simplify notation, we will set
\[
\nu_{0}=\nu+\nu_{r},\quad c_{1}=c_{a}+c_{d},\quad c_{2}=c_{0}+c_{d}-c_{a}.
\]
Furthermore, there is a slight difference in two and three dimensions.
Namely, if $d=2$, we assume that the velocity component in the $z$-direction
is zero and the angular velocity is parallel to the $z$-axis \citep{Grzegorz1999}.
That is, $\boldsymbol{u}=\left(u_{1},u_{2},0\right),\boldsymbol{w}=\left(0,0,w\right)$.

The MNS equations were first introduced by Eringen \citep{Eringen1966}
to describe the evolution of an incompressible fluid whose material
particles possess both translational and rotational motions. The novelties
of this system are to reflect the effects of microstructure on the
fluid via a microscopic dissipative evolution equations for the angular
momentum. Thus, this model is often used to describe the motion of
blood, certain lubricants, liquid crystals, ferrofluids, and some
polymeric fluids \citep{Grzegorz1999,Rosensweig1964,Nochetto2016a}. Given the significant role
it played in the microfluids, numerical solving of the MNS system
has drawn a considerable amount of attention. A penalty projection
method is proposed and optimal error estimates are proved in \citep{Elva2008}.
In \citep{Nochetto2014}, Nochetto et al. proposed and analyzed first-order
and second-order semi-implicit fully-discrete schemes.  These schemes decouple 
linear velocity computation from angular velocity computation, while being energy-stable. Later, Salgado
further adopted the fractional time stepping technique to decouple the
computation of pressure and velocity and proved the rigorous error
estimates in \citep{Salgado2015}. In these works, the nonlinear terms
are treated either implicitly or semi-implicitly so that one needs
to solve a nonlinear system or a linear system with variable coefficients
at each time step. It is desirable to treat the nonlinear term explicitly while maintaining energy stability. With such treatment,
the schemes only require the solution of linear system with constant
coefficients upon discretization, which are very efficient. 

In recent years, SAV based schemes have attracted much attention due to their efficiency, flexibility and accuracy. The main idea is to introduce auxiliary variables to preserve the property of energy decay. Several classes of energy stable numerical schemes have been developed for many dissipative
systems, like gradient flows \citep{Liu2021,Shen2018a,Shen2018,Shen2019}, NS equations \citep{Li2020,Li2020b,Lin2019},
magnetohydrodynamic equations \citep{Li2021,Yang2021m,Zhang2022a} and Cahn-Hilliard-Navier-Stokes
equations \citep{Li2022x,Li2020a,Yang2021c}. In particular, Shen et al. \citep{Huang2022a}
proposed a new class of efficient IMEX BDF$k$($1\le k\le5$) schemes
combined with a SAV approach for general dissipative systems. The
distinct advantages are that their higher-order versions are also unconditionally
energy stable and only require solving one decoupled linear system
with constant coefficients at each time step. 

For the MNS equations considered in this article, the energy structure is an inequality rather than an equality like many dissipative systems. This fact makes the energy-equality based
approaches \citep{Liu2021,Huang2022a} fail. Thus, it is not trivial
to construct efficient SAV schemes for such systems. The aim of this
work is to extend the approach proposed in \citep{Li2020b} to the
MNS equations. Our main contributions are three-folds: 
\begin{enumerate}
\item We propose a decoupled, linear and first-order scheme for the MNS
equations by combining the SAV approach for the convective terms and
some subtle IMEX treatments for the coupling terms. The scheme only
requires solving a sequence of differential equations with constant
coefficients at each time step so it is very efficient and easy to
implement. 
\item We establish rigorous unconditional energy stability and error analysis
for the proposed scheme in two dimensions. 
\item We provide some numerical experiments to confirm the predictions of
the theory and demonstrate the efficiency of the scheme.
\end{enumerate}
Compared to the Navier-Stokes equations, the error analysis for the
MNS equations is much more involved due to the coupling terms. It
is remarked that the present idea can be applies to the Boussinesq
equations and ferrohydrodynamics equations. 

The rest of this paper is organized as follows. In Section \ref{sec:Pre}, we introduce
some notations and present the energy estimate fo the MNS equations.
In Section \ref{sec:Scheme}, we propose the Euler IMEX-SAV scheme
and prove the unconditional stability. In Section \ref{sec:Error},
we carry out a rigorous error analysis for the proposed scheme in
two dimensions. In Section \ref{sec:Num}, we present some
numerical experiments. In Section \ref{sec:Conclud}, we conclude
with a few remarks.

\section{Preliminaries\label{sec:Pre}}

We start by introducing some notations and spaces. As usual, the inner
product and norm in $L^{2}(\Omega)$ are denoted by $(\cdot,\cdot)$
and $\left\Vert \cdot\right\Vert $, respectively. Let $W^{m,p}(\Omega)$
stand for the standard Sobolev spaces equipped with the standard Sobolev
norms $\left\Vert \cdot\right\Vert _{m,p}$. For $p=2$, we write
$H^{m}(\Omega)$ for $W^{m,2}(\Omega)$ and its corresponding norm
is $\left\Vert \cdot\right\Vert _{m}$. For a given Sobolev space
$X$, we write $L^{q}(0,T;X)$ for the Bochner space. Throughout the paper, we use $C$ to denote generic positive constants independent
of the discretization parameters, which may take different values
at different places.

For convenience, we introduce some notations for function spaces
\[
\boldsymbol{X}\coloneqq\boldsymbol{H}_{0}^{1}(\Omega),\quad
\boldsymbol{V}\coloneqq\left\{ \boldsymbol{v}\in\boldsymbol{X}:\nabla\cdot\boldsymbol{v}=0\right\}.
\]
The following equation for the curl operator will be repeatedly used in our analysis
\[
(\nabla\times\boldsymbol{w},\boldsymbol{u})=(\boldsymbol{w},\nabla\times\boldsymbol{u}),\quad\forall\boldsymbol{u},\boldsymbol{w}\in\boldsymbol{X}.
\]
Moreover, we recall that the following orthogonal decomposition
of $\boldsymbol{X}$ ,
\[
\left\Vert \nabla\boldsymbol{u}\right\Vert ^{2}=\left\Vert \nabla\times\boldsymbol{u}\right\Vert ^{2}+\left\Vert \nabla\cdot\boldsymbol{u}\right\Vert ^{2},\quad\forall\boldsymbol{u}\in\boldsymbol{X},
\]
which implies
\begin{equation}
\left\Vert \nabla\times\boldsymbol{u}\right\Vert \le\left\Vert \nabla\boldsymbol{u}\right\Vert, \quad\left\Vert \nabla\cdot\boldsymbol{u}\right\Vert \le\left\Vert \nabla\boldsymbol{u}\right\Vert, \quad\forall\boldsymbol{u}\in\boldsymbol{X}.\label{eq:norm curl}
\end{equation}

To deal with the convection terms in \eqref{eq:modelu} and \eqref{eq:modelw}, we define the following trilinear form, 
\[
\begin{aligned}b\end{aligned}
(\boldsymbol{u},\boldsymbol{v},\boldsymbol{w})=\left(\boldsymbol{u}\cdot\nabla\boldsymbol{v},\boldsymbol{w}\right).
\]
It is easy to see that the trilinear form $b(\cdot,\cdot,\cdot)$
is a skew-symmetric with respect to its last two arguments,
\begin{equation}
\begin{aligned}b\end{aligned}
(\boldsymbol{u},\boldsymbol{v},\boldsymbol{w})=-b(\boldsymbol{u},\boldsymbol{w},\boldsymbol{v}),\quad\forall\boldsymbol{u}\in\boldsymbol{V},\quad\boldsymbol{v},\boldsymbol{w}\in\boldsymbol{X},\label{eq:e_skew-symmetric1}
\end{equation}
and 
\begin{equation}
\begin{aligned}b\end{aligned}
(\boldsymbol{u},\boldsymbol{v},\boldsymbol{v})=0,\quad\forall\boldsymbol{u}\in\boldsymbol{V},\quad\boldsymbol{v}\in\boldsymbol{X}.\label{eq:e_skew-symmetric2}
\end{equation}

To end this section, we give the basic formal energy estimates for
the model \eqref{eq:MNS_model}. By taking the $\boldsymbol{L}^{2}$-inner
product of \eqref{eq:modelu} with $\boldsymbol{u}$, using the integration
by parts and \eqref{eq:modeldivu}, we get
\[
\frac{1}{2}\frac{d}{dt}\left\Vert \boldsymbol{u}\right\Vert ^{2}+\nu_{0}\left\Vert \nabla\boldsymbol{u}\right\Vert ^{2}+\left(\boldsymbol{u}\cdot\nabla\boldsymbol{u},\boldsymbol{u}\right)=2\nu_{r}(\nabla\times\boldsymbol{u},\boldsymbol{w}).
\]
Taking the $\boldsymbol{L}^{2}$-inner product of \eqref{eq:modelw}
with $\boldsymbol{w}$, and using the integration by parts, we have
\[
\frac{\jmath}{2}\frac{d}{dt}\left\Vert \boldsymbol{w}\right\Vert ^{2}+c_{1}\left\Vert \nabla\boldsymbol{w}\right\Vert ^{2}+\jmath\left(\boldsymbol{u}\cdot\nabla\boldsymbol{w},\boldsymbol{w}\right)+c_{2}\left\Vert \nabla\cdot\boldsymbol{w}\right\Vert ^{2}+4\nu_{r}\left\Vert \boldsymbol{w}\right\Vert ^{2}=2\nu_{r}(\nabla\times\boldsymbol{u},\boldsymbol{w}).
\]
Adding both ensuing equations and using \eqref{eq:e_skew-symmetric2},
we obtain
\begin{equation}
\frac{d}{dt}\left(\frac{1}{2}\left\Vert \boldsymbol{u}\right\Vert ^{2}+\frac{\jmath}{2}\left\Vert \boldsymbol{w}\right\Vert ^{2}\right)+\nu_{0}\left\Vert \nabla\boldsymbol{u}\right\Vert ^{2}+c_{1}\left\Vert \nabla\boldsymbol{w}\right\Vert ^{2}+c_{2}\left\Vert \nabla\cdot\boldsymbol{w}\right\Vert ^{2}+4\nu_{r}\left\Vert \boldsymbol{w}\right\Vert ^{2}=4\nu_{r}(\nabla\times\boldsymbol{u},\boldsymbol{w}).\label{eq:energy}
\end{equation}
Invoking with the Cauchy-Schwarz inequality, Young inequality and
\eqref{eq:norm curl}, the right hand side of \eqref{eq:energy} can
be estimated as 
\begin{equation}
4\nu_{r}(\nabla\times\boldsymbol{u},\boldsymbol{w})\le4\nu_{r}\left\Vert \nabla\times\boldsymbol{u}\right\Vert \left\Vert \boldsymbol{w}\right\Vert \le\nu_{r}\left\Vert \nabla\times\boldsymbol{u}\right\Vert ^{2}+4\nu_{r}\left\Vert \boldsymbol{w}\right\Vert ^{2}\le\nu_{r}\left\Vert \nabla\boldsymbol{u}\right\Vert ^{2}+4\nu_{r}\left\Vert \boldsymbol{w}\right\Vert ^{2}.\label{eq:energy_rhs}
\end{equation}
Inserting \eqref{eq:energy_rhs} into \eqref{eq:energy}, we have
\[
\frac{d}{dt}\left(\frac{1}{2}\left\Vert \boldsymbol{u}\right\Vert ^{2}+\frac{\jmath}{2}\left\Vert \boldsymbol{w}\right\Vert ^{2}\right)+\nu\left\Vert \nabla\boldsymbol{u}\right\Vert ^{2}+c_{1}\left\Vert \nabla\boldsymbol{w}\right\Vert ^{2}+c_{2}\left\Vert \nabla\cdot\boldsymbol{w}\right\Vert ^{2}\le0.
\]
Note that the energy dissipation law for the MNS equations is an
inequality rather than an equality, which is different from the one
for many other systems. The main reason is that the coupling terms
can not be canceled automatically in the process of deriving the energy
estimates. We further find that the nonlinear terms do not contribute
to the energy due to the skew-symmetric property in the above proof
to obtain the law of energy dissipation. The unique \textquotedbl{}
zero-energy-contribution\textquotedbl{} property will be used to design
efficient numerical schemes.

\section{Numerical scheme\label{sec:Scheme}}

In this section, we propose a Euler IMEX scheme based on the SAV approach
for the MNS equations and show that it is unconditionally energy
stable.

Inspired by the recent works \citep{Li2020b}, we introduce a scalar
auxiliary variable
\begin{equation}
\begin{aligned}q\end{aligned}
(t)\coloneqq\exp(-\frac{t}{T}).\label{e_definition of q}
\end{equation}
Noticing that $q(t)/\exp\left(-\frac{t}{T}\right)=1$, we reformulate
\eqref{eq:modelu} and \eqref{eq:modelw} into the equivalent forms
as follows,
\begin{equation}
\boldsymbol{u}_{t}+\frac{q}{\exp(-\frac{t}{T})}\boldsymbol{u}\cdot\nabla\boldsymbol{u}-\nu_{0}\Delta\boldsymbol{u}+\nabla p-2\nu_{r}\nabla\times\boldsymbol{w}=\boldsymbol{0},\label{eq:equiu}
\end{equation}
and
\begin{equation}
\jmath\boldsymbol{w}_{t}+\jmath\frac{q}{\exp(-\frac{t}{T})}\boldsymbol{u}\cdot\nabla\boldsymbol{w}-c_{1}\Delta\boldsymbol{w}-c_{2}\nabla\nabla\cdot\boldsymbol{w}+4\nu_{r}\boldsymbol{w}-2\nu_{r}\nabla\times\boldsymbol{u}=\boldsymbol{0}.\label{eq:equiw}
\end{equation}
Differentiating \eqref{e_definition of q} and using \eqref{eq:e_skew-symmetric2},
we have 
\begin{equation}
\frac{dq}{dt}=-\frac{1}{T}q+\frac{1}{\exp\left(-\frac{t}{T}\right)}\left((\boldsymbol{u}\cdot\nabla\boldsymbol{u},\boldsymbol{u})+\jmath\left(\boldsymbol{u}\cdot\nabla\boldsymbol{w},\boldsymbol{w}\right)\right).\label{eq:equiq}
\end{equation}
The last term in this equation is added to balance the nonlinear terms
in \eqref{eq:equiu} and \eqref{eq:equiw} in the discretized case.
Combining \eqref{eq:equiu}-\eqref{eq:equiq}, we recast the original
MNS equations as:
\begin{subequations}
\begin{align}
\boldsymbol{u}_{t}+\frac{q}{\exp(-\frac{t}{T})}\boldsymbol{u}\cdot\nabla\boldsymbol{u}-\nu_{0}\Delta\boldsymbol{u}+\nabla p-2\nu_{r}\nabla\times\boldsymbol{w} & =\boldsymbol{0},\label{eq:SAVmodelu}\\
\nabla\cdot\boldsymbol{u} & =0,\label{eq:SAVmodeldivu}\\
\jmath\boldsymbol{w}_{t}+\jmath\frac{q}{\exp(-\frac{t}{T})}\boldsymbol{u}\cdot\nabla\boldsymbol{w}-c_{1}\Delta\boldsymbol{w}-c_{2}\nabla\nabla\cdot\boldsymbol{w}+4\nu_{r}\boldsymbol{w}-2\nu_{r}\nabla\times\boldsymbol{u} & =\boldsymbol{0},\label{eq:SAVmodelw}\\
\frac{dq}{dt}+\frac{1}{T}q-\frac{1}{\exp\left(-\frac{t}{T}\right)}\left((\boldsymbol{u}\cdot\nabla\boldsymbol{u},\boldsymbol{u})+\jmath\left(\boldsymbol{u}\cdot\nabla\boldsymbol{w},\boldsymbol{w}\right)\right) & =0.\label{eq:SAVmodelq}
\end{align}
\label{eq:MNS_modelSAV}
\end{subequations}
It is clear that provided with $q(0)=1$, the exact solution of \eqref{eq:SAVmodelq}
is given by \eqref{e_definition of q}. Therefore, the above
system is equivalent to the original system. Note that the SAV $q(t)$ is related to the nonlinear
part of the free energy in the original SAV approach. However, the SAV $q(t)$ in this paper is purely artificial, which will allow us to construct unconditional energy stable schemes with
fully explicit treatment of the nonlinear terms.
\begin{thm}
\label{thm:energy_SAV}The expanded system \eqref{eq:MNS_modelSAV}
admits the following energy estimate,
\begin{align}
\frac{d}{dt}\left(\frac{1}{2}\left\Vert \boldsymbol{u}\right\Vert ^{2}+\frac{\jmath}{2}\left\Vert \boldsymbol{w}\right\Vert ^{2}+\frac{1}{2}|q|^{2}\right)+\nu\left\Vert \nabla\boldsymbol{u}\right\Vert ^{2}+c_{1}\left\Vert \nabla\boldsymbol{w}\right\Vert ^{2}+c_{2}\left\Vert \nabla\cdot\boldsymbol{w}\right\Vert ^{2}+\frac{1}{T}|q|^{2} & \le0.\label{eq:engdiss2}
\end{align}
\end{thm}
\begin{proof}
Taking the $\boldsymbol{L}^{2}$-inner product of \eqref{eq:SAVmodelu}
with $\boldsymbol{u}$, using the integration by parts and \eqref{eq:SAVmodeldivu},
we get
\begin{equation}
\frac{1}{2}\frac{d}{dt}\left\Vert \boldsymbol{u}\right\Vert ^{2}+\nu_{0}\left\Vert \nabla\boldsymbol{u}\right\Vert ^{2}+\frac{q}{\exp\left(-\frac{t}{T}\right)}\left(\boldsymbol{u}\cdot\nabla\boldsymbol{u},\boldsymbol{u}\right)=2\nu_{r}(\nabla\times\boldsymbol{u},\boldsymbol{w}).\label{eq:energy_u}
\end{equation}
Taking the $\boldsymbol{L}^{2}$-inner product of \eqref{eq:SAVmodelw}
with $\boldsymbol{w}$, and using the integration by parts, we have
\begin{equation}
\frac{\jmath}{2}\frac{d}{dt}\left\Vert \boldsymbol{w}\right\Vert ^{2}+c_{1}\left\Vert \nabla\boldsymbol{w}\right\Vert ^{2}+\jmath\frac{q}{\exp(-\frac{t}{T})}\left(\boldsymbol{u}\cdot\nabla\boldsymbol{w},\boldsymbol{w}\right)+c_{2}\left\Vert \nabla\cdot\boldsymbol{w}\right\Vert ^{2}+4\nu_{r}\left\Vert \boldsymbol{w}\right\Vert ^{2}=2\nu_{r}(\nabla\times\boldsymbol{u},\boldsymbol{w}).\label{eq:energy_w}
\end{equation}
Multiplying \eqref{eq:SAVmodelq} with $q$, we obtain 
\begin{equation}
\frac{1}{2}\frac{d}{dt}\left|q\right|^{2}+\frac{1}{T}\left|q\right|^{2}-\frac{q}{\exp\left(-\frac{t}{T}\right)}\left(\left(\boldsymbol{u}\cdot\nabla\boldsymbol{u},\boldsymbol{u}\right)+\jmath\left(\boldsymbol{u}\cdot\nabla\boldsymbol{w},\boldsymbol{w}\right)\right)=0.\label{eq:energy_q}
\end{equation}
By combining \eqref{eq:energy_u}-\eqref{eq:energy_q}, we derive
\begin{equation}
\frac{d}{dt}\left(\frac{1}{2}\left\Vert \boldsymbol{u}\right\Vert ^{2}+\frac{\jmath}{2}\left\Vert \boldsymbol{w}\right\Vert ^{2}+\frac{1}{2}\left|q\right|^{2}\right)+\nu_{0}\left\Vert \nabla\boldsymbol{u}\right\Vert ^{2}+c_{1}\left\Vert \nabla\boldsymbol{w}\right\Vert ^{2}+c_{2}\left\Vert \nabla\cdot\boldsymbol{w}\right\Vert ^{2}+4\nu_{r}\left\Vert \boldsymbol{w}\right\Vert ^{2}+\frac{1}{T}\left|q\right|^{2}=4\nu_{r}(\nabla\times\boldsymbol{u},\boldsymbol{w}).\label{eq:energy_sum}
\end{equation}
We finish the proof by using the estimate \eqref{eq:energy_rhs}.
\end{proof}

\begin{rem}
In this paper, the scalar auxiliary variable is only a time-dependent function $q(t)=\exp(-t/T)$ not a energy-related function. With this treatment, the algebraic equation for the scalar auxiliary variable is linear and unisolvent. Moreover, the ordinary differential equation for $q(t)$ is linear and dissipative, which makes our error estimates easier. In fact, the scalar auxiliary variable of this type admits a general form, $q(t)=C_{q,0}\exp(-C_{q,1}t/T)$ with $C_{q,0}\ne0$ and $C_{q,1}\ge0$. We refer to \citep{Zhang2022} for more details about this extension.
\end{rem}

\subsection{The SAV scheme}

Let $\left\{ t_{n}=n\tau:\,n=0,1,\cdots,N\right\} ,\tau=T/N,$ be
an equidistant partition of the time interval $[0,T].$ We denote
$(\cdot)^{n}$ as the variable $(\cdot)$ at time step $n.$ For any
function $v$, define 
\[
\delta_{t}v^{n+1}=\frac{v^{n+1}-v^{n}}{\tau}.
\]

Combining the backward Euler method and some delicate implicit/explicit
treatments for coupling terms, we propose a first-order SAV scheme
for solving the system \eqref{eq:MNS_modelSAV} as follows. Given
the initial conditions $\boldsymbol{u}^{0}$, $\boldsymbol{w}^{0}$
and $q^{0}$, compute $\left(\boldsymbol{u}^{n+1},p^{n+1},\boldsymbol{w}^{n+1},q^{n+1}\right)$,
$n=0,1,\cdots,N-1$ by 
\begin{subequations}
\begin{align}
\delta_{t}\boldsymbol{u}^{n+1}+\frac{q^{n+1}}{\exp(-\frac{t^{n+1}}{T})}\boldsymbol{u}^{n}\cdot\nabla\boldsymbol{u}^{n}-\nu_{0}\Delta\boldsymbol{u}^{n+1}+\nabla p^{n+1} & =2\nu_{r}\nabla\times\boldsymbol{w}^{n},\label{weakh:u}\\
\mathrm{div}\boldsymbol{u}^{n+1} & =0,\label{weakh:divu}\\
\jmath\delta_{t}\boldsymbol{w}^{n+1}+\jmath\frac{q^{n+1}}{\exp(-\frac{t^{n+1}}{T})}\boldsymbol{u}^{n}\cdot\nabla\boldsymbol{w}^{n}-c_{1}\Delta\boldsymbol{w}^{n+1}-c_{2}\nabla\nabla\cdot\boldsymbol{w}^{n+1}+4\nu_{r}\boldsymbol{w}^{n+1} & =2\nu_{r}\nabla\times\boldsymbol{u}^{n+1}\label{weakh:w}\\
\delta_{t}q^{n+1}+\frac{1}{T}q^{n+1}-\frac{1}{\exp\left(-\frac{t^{n+1}}{T}\right)}\left(\left(\boldsymbol{u}^{n}\cdot\nabla\boldsymbol{u}^{n},\boldsymbol{u}^{n+1}\right)+\jmath\left(\boldsymbol{u}^{n}\cdot\nabla\boldsymbol{w}^{n},\boldsymbol{w}^{n+1}\right)\right) & =0.\label{weakh:q}
\end{align}
\label{eq:SAVscheme}
\end{subequations}

Before giving further stability estimates, we first elaborate on how
to implement the proposed scheme efficiently. Since the auxiliary
variable $q(t)$ is a scalar number rather than a field function,
we can solve the nonlocally coupled scheme in a decoupled fashion.
Denote
\begin{equation}
S^{n+1}\coloneqq q^{n+1}\exp\left(\frac{t^{n+1}}{T}\right).\label{eq:q}
\end{equation}
We rewrite the first three equations in \eqref{eq:SAVscheme} into
\begin{align*}
\delta_{t}\boldsymbol{u}^{n+1}-\nu_{0}\Delta\boldsymbol{u}^{n+1}+\nabla p^{n+1} & =2\nu_{r}\nabla\times\boldsymbol{w}^{n}-S^{n+1}\boldsymbol{u}^{n}\cdot\nabla\boldsymbol{u}^{n},\\
\mathrm{div}\boldsymbol{u}^{n+1} & =0,\\
\jmath\delta_{t}\boldsymbol{w}^{n+1}-c_{1}\Delta\boldsymbol{w}^{n+1}-c_{2}\nabla\nabla\cdot\boldsymbol{w}^{n+1}+4\nu_{r}\boldsymbol{w}^{n+1}-2\nu_{r}\nabla\times\boldsymbol{u}^{n+1} & =-\jmath S^{n+1}(\boldsymbol{u}^{n}\cdot\nabla)\boldsymbol{w}^{n}.
\end{align*}
Barring the unknown scalar number $S^{n+1}$, they are linear equations
with respect to $\boldsymbol{u}^{n+1}$ and $\boldsymbol{w}^{n+1}$.
Inspired by the work in \citep{Lin2019}, we define two field functions
($\boldsymbol{u}_{i}^{n+1},p_{i}^{n+1},\boldsymbol{w}_{i}^{n+1}$), $i=1,2$, as solutions to the following
two problems:
\begin{align}
\frac{\boldsymbol{u}_{1}^{n+1}-\boldsymbol{u}^{n}}{\tau}-\nu_{0}\Delta\boldsymbol{u}_{1}^{n+1}+\nabla p_{1}^{n+1} & =2\nu_{r}\nabla\times\boldsymbol{w}^{n},\label{eq:weakh11u}\\
\mathrm{div}\boldsymbol{u}_{1}^{n+1} & =0,\label{eq:weakh11divu}\\
\jmath\frac{\boldsymbol{w}_{1}^{n+1}-\boldsymbol{w}^{n}}{\tau}-c_{1}\Delta\boldsymbol{w}_{1}^{n+1}-c_{2}\nabla\nabla\cdot\boldsymbol{w}_{1}^{n+1}+4\nu_{r}\boldsymbol{w}_{1}^{n+1}-2\nu_{r}\nabla\times\boldsymbol{u}_{1}^{n+1} & =0.\label{eq:weakh11w}
\end{align}
 and 
\begin{align}
\frac{\boldsymbol{u}_{2}^{n+1}}{\tau}-\nu_{0}\Delta\boldsymbol{u}_{2}^{n+1}+\nabla p_{2}^{n+1} & =-\boldsymbol{u}^{n}\cdot\nabla\boldsymbol{u}^{n},\label{eq:weakh12u}\\
\mathrm{div}\boldsymbol{u}_{2}^{n+1} & =0,\label{eq:weakh12divu}\\
\jmath\frac{\boldsymbol{w}_{2}^{n+1}}{\tau}-c_{1}\Delta\boldsymbol{w}_{2}^{n+1}-c_{2}\nabla\nabla\cdot\boldsymbol{w}_{2}^{n+1}+4\nu_{r}\boldsymbol{w}_{2}^{n+1}-2\nu_{r}\nabla\times\boldsymbol{u}_{2}^{n+1} & =-\jmath(\boldsymbol{u}^{n}\cdot\nabla)\boldsymbol{w}^{n}.\label{eq:weakh12w}
\end{align}
Then it is straightforward to verify that the solution of $\left(\boldsymbol{u},p,\boldsymbol{w}\right)$
to the scheme \eqref{eq:SAVscheme} is given by 
\begin{align}
\boldsymbol{u}^{n+1} & =\boldsymbol{u}_{1}^{n+1}+S^{n+1}\boldsymbol{u}_{2}^{n+1},\label{eq:soldecu}\\
p^{n+1} & =p_{1}^{n+1}+S^{n+1}p_{2}^{n+1},\label{eq:soldecp}\\
\boldsymbol{w}^{n+1} & =\boldsymbol{w}_{1}^{n+1}+S^{n+1}\boldsymbol{w}_{2}^{n+1},\label{eq:soldecw}
\end{align}
where $S^{n+1}$ is to be determined. Inserting \eqref{eq:soldecu}-\eqref{eq:soldecw}
into \eqref{weakh:q}, we have
\begin{equation}
\left(\frac{\tau+T}{\tau T}-\exp\left(\frac{2t^{n+1}}{T}\right)A_{2}\right)
\exp\left(-\frac{t^{n+1}}{T}\right)S^{n+1}=\exp\left(\frac{t^{n+1}}{T}\right)A_{1}+\frac{1}{\tau}q^{n},\label{eq:solq}
\end{equation}
where $A_{i}=(\boldsymbol{u}^{n}\cdot\nabla\boldsymbol{u}^{n},\boldsymbol{u}_{i}^{n+1})+\jmath(\boldsymbol{u}^{n}\cdot\nabla\boldsymbol{w}^{n},\boldsymbol{w}_{i}^{n+1}),\ i=1,2.$
From \eqref{eq:weakh12u}-\eqref{eq:weakh12w}, we know that the coefficient
of \eqref{eq:solq} is positive, 
\begin{align*}
\frac{\tau+T}{\tau T}-\exp\left(\frac{2t^{n+1}}{T}\right)A_{2} & =\frac{\tau+T}{\tau T}+\exp\left(\frac{2t^{n+1}}{T}\right)\left(\frac{\left\Vert \boldsymbol{u}_{2}^{n+1}\right\Vert ^{2}+\jmath\left\Vert \boldsymbol{w}_{2}^{n+1}\right\Vert ^{2}}{\tau}+\nu_{0}\left\Vert \nabla\boldsymbol{u}_{2}^{n+1}\right\Vert ^{2}+c_{1}\left\Vert \nabla\boldsymbol{w}_{2}^{n+1}\right\Vert ^{2}\right.\\
 & \quad\left.+c_{2}\left\Vert \nabla\cdot\boldsymbol{w}_{2}^{n+1}\right\Vert ^{2}+4\nu_{r}\left\Vert \boldsymbol{w}_{2}^{n+1}\right\Vert ^{2}-2\nu_{r}\left(\nabla\times\boldsymbol{u}_{2}^{n+1},\boldsymbol{w}_{2}^{n+1}\right)\right)\\
 & \ge\frac{\tau+T}{\tau T}+\exp\left(\frac{2t^{n+1}}{T}\right)\left(\frac{\left\Vert \boldsymbol{u}_{2}^{n+1}\right\Vert ^{2}+\jmath\left\Vert \boldsymbol{w}_{2}^{n+1}\right\Vert ^{2}}{\tau}+\nu\left\Vert \nabla\boldsymbol{u}_{2}^{n+1}\right\Vert ^{2}+c_{1}\left\Vert \nabla\boldsymbol{w}_{2}^{n+1}\right\Vert ^{2}\right.\\
 & \quad\left.+c_{2}\left\Vert \nabla\cdot\boldsymbol{w}_{2}^{n+1}\right\Vert ^{2}+3\nu_{r}\left\Vert \boldsymbol{w}_{2}^{n+1}\right\Vert ^{2}\right)\\
 & >0.
\end{align*}
This implies the existence and uniqueness of $S^{n+1}$. Thus, we arrive
at the final solution algorithm. It involves the following three steps:
\begin{enumerate}
\item Solve $\left(\boldsymbol{u}_{i}^{n+1},p_{i}^{n+1},\boldsymbol{w}_{i}^{n+1}\right),\,i=1,2$
by two substeps:
\begin{enumerate}
\item Solve equations \eqref{eq:weakh11u}-\eqref{eq:weakh11divu} and \eqref{eq:weakh12u}-\eqref{eq:weakh12divu}
for $\left(\boldsymbol{u}_{i}^{n+1},p_{i}^{n+1}\right),\,i=1,2$ ,
\item Solve equations \eqref{eq:weakh11w} and \eqref{eq:weakh12w} for
$\boldsymbol{w}_{i}^{n+1},\,i=1,2$.
\end{enumerate}
\item Solve equation \eqref{eq:solq} for $S^{n+1}$.
\item Compute $\left(\boldsymbol{u}^{n+1},p^{n+1},\boldsymbol{w}^{n+1}\right)$
by \eqref{eq:soldecu}-\eqref{eq:soldecw}, compute $q^{n+1}$ by
\eqref{eq:q}.
\end{enumerate}
To conclude, we only need to solve two generalized Stokes equations,
and two elliptic equations with constant coefficients plus a purely
linear algebraic equation at each time step. Therefore, the scheme is
quite efficient in the implementation.
\begin{rem}
In this paper, we decouple the NS equations and the angular momentum
equation by time-lagging of the angular velocity in \eqref{weakh:u}.
Therefore, we can first solve the NS equations and then solve the angular momentum
equation in succession. One can further lag the linear velocity in
\eqref{weakh:w} to solve the NS equations and the angular momentum
equation in parallel. The corresponding  theoretical analysis is much similar
to the one for the proposed scheme, we leave it for the interested readers.
\end{rem}

\begin{rem}
From the previous discussions, we can see that the SAV $q(t)$ can help us to design an unconditionally stable scheme. Meanwhile, it can decompose the discrete equations
into some sub-equations that can be solved efficiently.
In addition, it can also provide a practical strategy of adaptive time-stepping \cite{Shen2019,Huang2022a}. Generally speaking, when $S^{n+1}=q^{n+1}\exp(t^{n+1}/T)$ deviates from 1, the time step $\tau$ needs to be refined in order to maintain the accuracy. When $S^{n+1}$ stays close to 1, the time step $\tau$ can be relaxed. The detailed mechanism of the variable time step is an interesting work for future research.
\end{rem}

\subsection{Energy stability}
The unconditionally energy stability of the scheme is established
in the following theorem.
\begin{thm}
\label{thm_energy stability} The scheme \eqref{eq:SAVscheme} is
unconditionally stable in the sense that 
\begin{equation}
\begin{aligned}E\end{aligned}
^{n+1}-E^{n}\leq-\tau\nu\left\Vert \nabla\boldsymbol{u}^{n+1}\right\Vert ^{2}-\tau c_{1}\left\Vert \nabla\boldsymbol{w}^{n+1}\right\Vert ^{2}-\tau c_{2}\left\Vert \nabla\cdot\boldsymbol{w}^{n+1}\right\Vert ^{2}-\frac{\tau}{T}|q^{n+1}|^{2}\quad\forall n\geq0,\label{eq:discteng}
\end{equation}
where 
\[
E^{n+1}=\frac{1}{2}\left\Vert \boldsymbol{u}^{n+1}\right\Vert ^{2}+\frac{\jmath+4\tau\nu_{r}}{2}\left\Vert \boldsymbol{w}^{n+1}\right\Vert ^{2}+\frac{1}{2}|q^{n+1}|^{2}.
\]
\end{thm}
\begin{proof}
Taking the $\boldsymbol{L}^{2}$-inner product of \eqref{weakh:u}
with $\boldsymbol{u}^{n+1}$, using the identity 
\begin{equation}
\begin{aligned}(\end{aligned}
a-b,a)=\frac{1}{2}(|a|^{2}-|b|^{2}+|a-b|^{2}),\label{eq:e_identity_Euler}
\end{equation}
and \eqref{weakh:divu}, we get 
\begin{equation}
\frac{\left\Vert \boldsymbol{u}^{n+1}\right\Vert ^{2}-\left\Vert \boldsymbol{u}^{n}\right\Vert ^{2}+\left\Vert \boldsymbol{u}^{n+1}-\boldsymbol{u}^{n}\right\Vert ^{2}}{2\tau}+\nu_{0}\left\Vert \nabla\boldsymbol{u}^{n+1}\right\Vert ^{2}=2\nu_{r}\left(\nabla\times\boldsymbol{w}^{n},\boldsymbol{u}^{n+1}\right)-\frac{q^{n+1}}{\exp(-\frac{t^{n+1}}{T})}\left(\boldsymbol{u}^{n}\cdot\nabla\boldsymbol{u}^{n},\boldsymbol{u}^{n+1}\right).\label{eq:e_stability_first_u}
\end{equation}
Taking the $\boldsymbol{L}^{2}$-inner product of \eqref{weakh:w}
with $\boldsymbol{w}^{n+1}$ and using the identity \eqref{eq:e_identity_Euler}
again, we obtain 
\begin{align}
&\jmath\frac{\left\Vert \boldsymbol{w}^{n+1}\right\Vert ^{2}-\left\Vert \boldsymbol{w}^{n}\right\Vert ^{2}+\left\Vert \boldsymbol{w}^{n+1}-\boldsymbol{w}^{n}\right\Vert ^{2}}{2\tau}+c_{1}\left\Vert \nabla\boldsymbol{w}^{n+1}\right\Vert ^{2}+c_{2}\left\Vert \nabla\cdot\boldsymbol{w}^{n+1}\right\Vert ^{2}+4\nu_{r}\left\Vert \boldsymbol{w}^{n+1}\right\Vert ^{2}\nonumber\\
&=2\nu_{r}\left(\nabla\times\boldsymbol{u}^{n+1},\boldsymbol{w}^{n+1}\right)-\jmath\frac{q^{n+1}}{\exp(-\frac{t^{n+1}}{T})}\left(\boldsymbol{u}^{n}\cdot\nabla\boldsymbol{w}^{n},\boldsymbol{w}^{n+1}\right).\label{eq:e_stability_first_w}
\end{align}
Multiplying \eqref{weakh:q} by $q^{n+1}$ and using the identity
\eqref{eq:e_identity_Euler} again, we have
\begin{equation}
\frac{|q^{n+1}|^{2}-|q^{n}|^{2}+|q^{n+1}-q^{n}|^{2}}{2\tau}+\frac{1}{T}|q^{n+1}|^{2}=\frac{q^{n+1}}{\exp(-\frac{t^{n+1}}{T})}\left(\left(\boldsymbol{u}^{n}\cdot\nabla\boldsymbol{u}^{n},\boldsymbol{u}^{n+1}\right)+\jmath\left(\boldsymbol{u}^{n}\cdot\nabla\boldsymbol{w}^{n},\boldsymbol{w}^{n+1}\right)\right).\label{eq:e_stability_first_q}
\end{equation}
By taking the summations of \eqref{eq:e_stability_first_u}-\eqref{eq:e_stability_first_q},
we get
\begin{align}
&\frac{\left\Vert \boldsymbol{u}^{n+1}\right\Vert ^{2}-\left\Vert \boldsymbol{u}^{n}\right\Vert ^{2}+\left\Vert \boldsymbol{u}^{n+1}-\boldsymbol{u}^{n}\right\Vert ^{2}}{2\tau}+\jmath\frac{\left\Vert \boldsymbol{w}^{n+1}\right\Vert ^{2}-\left\Vert \boldsymbol{w}^{n}\right\Vert ^{2}+\left\Vert \boldsymbol{w}^{n+1}-\boldsymbol{w}^{n}\right\Vert ^{2}}{2\tau} \nonumber \\
&\quad+\frac{|q^{n+1}|^{2}-|q^{n}|^{2}+|q^{n+1}-q^{n}|^{2}}{2\tau}+\nu_{0}\left\Vert \nabla\boldsymbol{u}^{n+1}\right\Vert ^{2}+c_{1}\left\Vert \nabla\boldsymbol{w}^{n+1}\right\Vert ^{2}\nonumber \\
&\quad+c_{2}\left\Vert \nabla\cdot\boldsymbol{w}^{n+1}\right\Vert ^{2}+4\nu_{r}\left\Vert \boldsymbol{w}^{n+1}\right\Vert ^{2}+\frac{1}{T}|q^{n+1}|^{2}\nonumber \\
&=2\nu_{r}\left(\nabla\times\boldsymbol{w}^{n},\boldsymbol{u}^{n+1}\right)+2\nu_{r}\left(\nabla\times\boldsymbol{u}^{n+1},\boldsymbol{w}^{n+1}\right).\label{eq:stability_sum}
\end{align}
Using Cauchy-Schwarz inequality and Young inequality, we derive the
right hand side of \eqref{eq:stability_sum} has the following estimate,
\begin{align}
2\nu_{r}\left(\nabla\times\boldsymbol{w}^{n},\boldsymbol{u}^{n+1}\right)+2\nu_{r}\left(\nabla\times\boldsymbol{u}^{n+1},\boldsymbol{w}^{n+1}\right) & \le2\nu_{r}\left\Vert \nabla\times\boldsymbol{u}^{n+1}\right\Vert \left\Vert \boldsymbol{w}^{n}\right\Vert +2\nu_{r}\left\Vert \nabla\times\boldsymbol{u}^{n+1}\right\Vert \left\Vert \boldsymbol{w}^{n+1}\right\Vert \nonumber \\
 & \le\nu_{r}\left\Vert \nabla\boldsymbol{u}^{n+1}\right\Vert ^{2}+2\nu_{r}\left\Vert \boldsymbol{w}^{n}\right\Vert ^{2}+2\nu_{r}\left\Vert \boldsymbol{w}^{n+1}\right\Vert ^{2}.\label{eq:stability_rhs}
\end{align}
Plugging \eqref{eq:stability_rhs} into \eqref{eq:stability_sum},
we gain the required estimate. The proof is thus complete.
\end{proof}
We observe that the discrete energy dissipation law \eqref{eq:discteng}
is an approximation of the continuous energy dissipation law \eqref{eq:engdiss2}. 
\begin{cor}[Stability]
 Let $\left(\boldsymbol{u}^{n},\boldsymbol{w}^{n},p^{n},q^{n}\right),\, n\ge0$
solve (\ref{eq:SAVscheme}). Then it satisfies the following stability
estimate for any $m\ge0$,
\begin{align*}
 & \left\Vert \boldsymbol{u}^{m}\right\Vert ^{2}+\left(\jmath+4\tau\nu_{r}\right)\left\Vert \boldsymbol{w}^{m}\right\Vert ^{2}+\left|q^{m}\right|^{2}\\
 &\quad+2\tau\sum_{n=0}^{m}\left(\nu\left\Vert \nabla\boldsymbol{u}^{n+1}\right\Vert ^{2}+c_{1}\left\Vert \nabla\boldsymbol{w}^{n+1}\right\Vert ^{2}+c_{2}\left\Vert \nabla\cdot\boldsymbol{w}^{n+1}\right\Vert ^{2}+\frac{1}{T}|q^{n+1}|^{2}\right)\\
 & \leq\left\Vert \boldsymbol{u}^{0}\right\Vert ^{2}+\left(\jmath+4\tau\nu_{r}\right)\left\Vert \boldsymbol{w}^{0}\right\Vert ^{2}+\left|q^{0}\right|^{2}.
\end{align*}
\end{cor}
\begin{proof}
By Theorem \ref{thm_energy stability}, summing up inequality \eqref{eq:discteng}
from $n=0$ to $m-1$, we obtain the stable bound.
\end{proof}
Based on this corollary, we can easily obtain the following uniform
bounds for any $m\ge0$,
\begin{align}
\left\Vert \boldsymbol{u}^{m}\right\Vert ^{2}+\jmath\left\Vert \boldsymbol{w}^{m}\right\Vert ^{2}+\left|q^{m}\right|^{2} & \le k_{1},\label{eq:u_uniform}\\
\tau\sum_{n=0}^{m}\left(\nu\left\Vert \nabla\boldsymbol{u}^{n}\right\Vert ^{2}+c_{1}\left\Vert \nabla\boldsymbol{w}^{n}\right\Vert ^{2}+c_{2}\left\Vert \nabla\cdot\boldsymbol{w}^{n}\right\Vert ^{2}+\frac{1}{T}|q^{n}|^{2}\right) & \le k_{2},\label{eq:gradu_uniform}
\end{align}
where the constants $k_{i}$ $(i=1,2)$ are independent of $\tau$. 

\section{Error Analysis\label{sec:Error}}

In this section, we give a rigorous error analysis for the scheme
\eqref{eq:SAVscheme} in two dimensions. We emphasize that
while the scheme can be used in three dimensions, the error
analysis can not be easily extended to three dimensions due
to some technical issues. Thus, we set $d=2$ in this section.

Denote the following error functions
\[
e_{\boldsymbol{u}}^{n}=\boldsymbol{u}^{n}-\boldsymbol{u}\left(t^{n}\right),\quad e_{p}^{n}=p^{n}-p\left(t^{n}\right),\quad e_{\boldsymbol{w}}^{n}=\boldsymbol{w}^{n}-\boldsymbol{w}\left(t^{n}\right),\quad e_{q}^{n}=q^{n}-q\left(t^{n}\right).
\]
Subtracting \eqref{eq:MNS_modelSAV} at $t^{n+1}$ from \eqref{eq:SAVscheme},
and noticing $q(t^{n+1})=\exp(-t^{n+1}/T)$, we get the following
error equations
\begin{subequations}
\begin{align}
\delta_{t}e_{\boldsymbol{u}}^{n+1}+\left(q^{n+1}\exp\left(\frac{t^{n+1}}{T}\right)\boldsymbol{u}^{n}\cdot\nabla\boldsymbol{u}^{n}-\boldsymbol{u}(t^{n+1})\cdot\nabla\boldsymbol{u}(t^{n+1})\right)\nonumber \\
-\nu_{0}\Delta e_{\boldsymbol{u}}^{n+1}+\nabla e_{p}^{n+1}-2\nu_{r}\left(\nabla\times\boldsymbol{w}^{n}-\nabla\times\boldsymbol{w}(t^{n+1})\right) & =R_{\boldsymbol{u}}^{n+1},\label{eq:SAVerroru}\\
\nabla\cdot e_{\boldsymbol{u}}^{n+1} & =0,\label{eq:SAVerrordivu}\\
\jmath\delta_{t}e_{\boldsymbol{w}}^{n+1}+\jmath\left(q^{n+1}\exp\left(\frac{t^{n+1}}{T}\right)\boldsymbol{u}^{n}\cdot\nabla\boldsymbol{w}^{n}-\boldsymbol{u}(t^{n+1})\cdot\nabla\boldsymbol{w}(t^{n+1})\right)\nonumber \\
-c_{1}\Delta e_{\boldsymbol{w}}^{n+1}-c_{2}\nabla\nabla\cdot e_{\boldsymbol{w}}^{n+1}+4\nu_{r}e_{\boldsymbol{w}}^{n+1}-2\nu_{r}\nabla\times e_{\boldsymbol{u}}^{n+1} & =R_{\boldsymbol{w}}^{n+1},\label{eq:SAVerrow}\\
\delta_{t}e_{q}^{n+1}+\frac{1}{T}e_{q}^{n+1}-\exp\left(-\frac{t^{n+1}}{T}\right)\left(\left(\boldsymbol{u}^{n}\cdot\nabla\boldsymbol{u}^{n},\boldsymbol{u}^{n+1}\right)-\left(\boldsymbol{u}(t^{n+1})\cdot\nabla\boldsymbol{u}(t^{n+1}),\boldsymbol{u}(t^{n+1})\right)\right)\nonumber \\
-\jmath\exp\left(-\frac{t^{n+1}}{T}\right)\left(\left(\boldsymbol{u}^{n}\cdot\nabla\boldsymbol{w}^{n},\boldsymbol{w}^{n+1}\right)-\left(\boldsymbol{u}(t^{n+1})\cdot\nabla\boldsymbol{w}(t^{n+1}),\boldsymbol{w}(t^{n+1})\right)\right) & =R_{q}^{n+1},\label{eq:SAVerrorq}
\end{align}
\label{MNS_errorSAV}
\end{subequations}
 where $R_{\boldsymbol{u}}^{n+1}$, $R_{\boldsymbol{w}}^{n+1}$ and
$R_{q}^{n+1}$ are the truncation errors, 
\[
R_{\boldsymbol{u}}^{n+1}:=\frac{1}{\tau}\int_{t^{n}}^{t^{n+1}}\left(s-t^{n}\right)\boldsymbol{u}_{tt}(s)ds,\, R_{\boldsymbol{w}}^{n+1}:=\frac{1}{\tau}\int_{t^{n}}^{t^{n+1}}\left(s-t^{n}\right)\boldsymbol{w}_{tt}(s)ds,\, R_{q}^{n+1}:=\frac{1}{\tau}\int_{t^{n}}^{t^{n+1}}\left(s-t^{n}\right)q_{tt}(s)ds.
\]

Let $P$ be the orthogonal projector in $\boldsymbol{L}^{2}(\Omega)$
onto $\boldsymbol{V}$, we define the Stokes operator $A$ by 
\[
A\boldsymbol{u}=-P\Delta\boldsymbol{u},\quad\forall\boldsymbol{u}\in D(A)=\boldsymbol{H}^{2}(\Omega)\cap\boldsymbol{X}.
\] 
The following estimates for the trilinear form $b(\cdot,\cdot,\cdot)$
will be used in our analysis \citep{Li2020b,Li2021,Temam1977,Temam1995}.
\begin{lem}
The following estimates of the trilinear form hold 
\begin{align}
b(\boldsymbol{u},\boldsymbol{v},\boldsymbol{w}) & \le C_{b,0}\left\Vert \nabla\boldsymbol{u}\right\Vert \left\Vert\nabla \boldsymbol{v}\right\Vert\left\Vert \nabla\boldsymbol{w}\right\Vert, \quad\forall\boldsymbol{u},\boldsymbol{v},\boldsymbol{w}\in\boldsymbol{X},\label{eq:e_estimate for trilinear form0}\\	
b(\boldsymbol{u},\boldsymbol{v},\boldsymbol{w}) & \le C_{b,1}\left\Vert \boldsymbol{u}\right\Vert \left\Vert \boldsymbol{v}\right\Vert _{2}\left\Vert \nabla\boldsymbol{w}\right\Vert, \quad\forall\boldsymbol{u},\boldsymbol{w}\in\boldsymbol{X},\quad\boldsymbol{v}\in\boldsymbol{X}\cap\boldsymbol{H}^{2}\left(\Omega\right),\label{eq:e_estimate for trilinear form}\\
b(\boldsymbol{u},\boldsymbol{v},\boldsymbol{w}) & \le C_{b,2}\left\Vert \boldsymbol{u}\right\Vert _{2}\left\Vert \boldsymbol{v}\right\Vert \left\Vert \nabla\boldsymbol{w}\right\Vert, \quad\forall\boldsymbol{v},\boldsymbol{w}\in\boldsymbol{X},\quad\boldsymbol{u}\in\boldsymbol{V}\cap\boldsymbol{H}^{2}\left(\Omega\right),\label{eq:e_estimate for trilinear form1}\\
b(\boldsymbol{u},\boldsymbol{v},\boldsymbol{w}) & \le C_{b,3}\left\Vert \nabla\boldsymbol{u}\right\Vert \left\Vert \boldsymbol{v}\right\Vert \left\Vert \boldsymbol{w}\right\Vert _{2},\quad\forall\boldsymbol{u}\in\boldsymbol{V},\quad\boldsymbol{v}\in\boldsymbol{X},\quad\boldsymbol{w}\in\boldsymbol{X}\cap\boldsymbol{H}^{2}\left(\Omega\right),\label{eq:e_estimate for trilinear form2}\\
b(\boldsymbol{u},\boldsymbol{v},\boldsymbol{w}) & \le C_{b,4}\left\Vert \boldsymbol{u}\right\Vert \left\Vert \nabla\boldsymbol{v}\right\Vert \left\Vert \boldsymbol{w}\right\Vert_{2},\quad\forall\boldsymbol{u},\boldsymbol{v}\in\boldsymbol{X},\quad\boldsymbol{w}\in\boldsymbol{X}\cap\boldsymbol{H}^{2}\left(\Omega\right),\label{eq:e_estimate for trilinear form3}\\
b(\boldsymbol{u},\boldsymbol{v},\boldsymbol{w}) & \le C_{b,5}\left\Vert \nabla\boldsymbol{u}\right\Vert \left\Vert \boldsymbol{v}\right\Vert _{2}\left\Vert \boldsymbol{w}\right\Vert, \quad\forall\boldsymbol{u},\boldsymbol{w}\in\boldsymbol{X},\quad\boldsymbol{v}\in\boldsymbol{X}\cap\boldsymbol{H}^{2}\left(\Omega\right).\label{eq:e_estimate for trilinear form4}
\end{align}
Moreover, for $d=2$, we have
\begin{align}
b(\boldsymbol{u},\boldsymbol{v},\boldsymbol{w}) & \leq C_{b,6}\left\Vert \nabla\boldsymbol{u}\right\Vert ^{1/2}\left\Vert \boldsymbol{u}\right\Vert ^{1/2}\left\Vert \nabla\boldsymbol{v}\right\Vert ^{1/2}\left\Vert \boldsymbol{v}\right\Vert ^{1/2}\left\Vert \nabla\boldsymbol{w}\right\Vert ,\quad\forall\boldsymbol{u}\in\boldsymbol{V},\quad\boldsymbol{v},\boldsymbol{w}\in\boldsymbol{X},\label{eq:e_estimate for trilinear form2d}\\
b(\boldsymbol{u},\boldsymbol{v},\boldsymbol{w}) & \leq C_{b,7}\left\Vert \nabla\boldsymbol{u}\right\Vert ^{1/2}\left\Vert \boldsymbol{u}\right\Vert ^{1/2}\left\Vert \nabla\boldsymbol{v}\right\Vert ^{1/2}\left\Vert A\boldsymbol{v}\right\Vert ^{1/2}\left\Vert \boldsymbol{w}\right\Vert ,\quad\forall\boldsymbol{v}\in\boldsymbol{X}\cap\boldsymbol{H}^{2}\left(\Omega\right),\quad\boldsymbol{u},\boldsymbol{w}\in\boldsymbol{X},\label{eq:e_estimate for trilinear form2d1}\\
b(\boldsymbol{u},\boldsymbol{v},\boldsymbol{w}) & \leq C_{b,8}\left\Vert \boldsymbol{u}\right\Vert ^{1/2}\left\Vert A\boldsymbol{u}\right\Vert ^{1/2}\left\Vert \nabla\boldsymbol{v}\right\Vert \left\Vert \boldsymbol{w}\right\Vert ,\quad\forall\boldsymbol{u}\in\boldsymbol{X}\cap\boldsymbol{H}^{2}\left(\Omega\right),\quad\boldsymbol{v},\boldsymbol{w}\in\boldsymbol{X}.\label{eq:e_estimate for trilinear form2d2}
\end{align}
\end{lem}
The following discrete version of the Gronwall
lemma will be frequently used \citep{Temam1977,John2016,Ladyzhenskaya1969}.
\begin{lem}
\label{lem: gronwall2} Let $a_{n},b_{n},c_{n}$, and $d_{n}$ be
four non-negative sequences satisfying 
\[
a_{m}+\tau\sum_{n=1}^{m}b_{n}\leq\tau\sum_{n=0}^{m-1}a_{n}d_{n}+\tau\sum_{n=0}^{m-1}c_{n}+C,\quad m\geq1,
\]
 where $C$ and $\tau$ are two positive constants. Then 
\[
a_{m}+\tau\sum_{n=1}^{m}b_{n}\leq\exp\left(\tau\sum_{n=0}^{m-1}d_{n}\right)\left(\tau\sum_{n=0}^{m-1}c_{n}+C\right),\quad m\geq1.
\]
\end{lem}

\subsection{Error estimates for the velocity and angular velocity}

In this subsection, we derive the following error estimates for the
velocity $\boldsymbol{u}$ and angular velocity $\boldsymbol{w}$.
\begin{thm}
\label{thm: error_estimate_uwq} Assume the exact solution satisfies  $\boldsymbol{u}\in H^{2}(0,T;\boldsymbol{H}^{-1}(\Omega))\bigcap H^{1}(0,T;\boldsymbol{H}^{2}(\Omega))\bigcap L^{\infty}(0,T;\boldsymbol{H}^{2}(\Omega))$,
and $\boldsymbol{w}\in H^{2}(0,T;\boldsymbol{H}^{-1}(\Omega))\bigcap H^{1}(0,T;\boldsymbol{H}^{2}(\Omega))\bigcap L^{\infty}(0,T;\boldsymbol{H}^{2}(\Omega))$,
then we have 
\begin{align}
 & \left\Vert e_{\boldsymbol{u}}^{m+1}\right\Vert ^{2}+\nu\tau\sum_{n=0}^{m}\left\Vert \nabla e_{\boldsymbol{u}}^{n+1}\right\Vert ^{2}+\left(\jmath+4\nu_r\tau\right)\left\Vert e_{\boldsymbol{w}}^{m+1}\right\Vert ^{2}+c_{1}\tau\sum_{n=0}^{m}\left\Vert \nabla e_{\boldsymbol{w}}^{n+1}\right\Vert ^{2}\nonumber \\
 & +2c_{2}\tau\sum_{n=0}^{m}\left\Vert \nabla\cdot e_{\boldsymbol{w}}^{n+1}\right\Vert ^{2}+|e_{q}^{m+1}|^{2}+\frac{\tau}{T}\sum_{n=0}^{m}|e_{q}^{n+1}|^{2}\leq C\tau^{2},\quad\forall\ 0\leq m\leq N-1.\label{eq:estuwq}
\end{align}
 
\end{thm}
The proof of the above theorem will be carried out with a sequence
of lemmas below.

First, we derive a estimate for the velocity error.
\begin{lem}
\label{lem: error_estimate_u} Under the assumptions of Theorem \ref{thm: error_estimate_uwq},
we have 
\begin{align}
 & \frac{\left\Vert e_{\boldsymbol{u}}^{n+1}\right\Vert ^{2}-\left\Vert e_{\boldsymbol{u}}^{n}\right\Vert ^{2}+\left\Vert e_{\boldsymbol{u}}^{n+1}-e_{\boldsymbol{u}}^{n}\right\Vert ^{2}}{2\tau}+\frac{\nu+\nu_{r}}{2}\left\Vert \nabla e_{\boldsymbol{u}}^{n+1}\right\Vert ^{2}\nonumber \\
 & \leq-\exp\left(\frac{t^{n+1}}{T}\right)e_{q}^{n+1}\left(\boldsymbol{u}^{n}\cdot\nabla\boldsymbol{u}^{n},e_{\boldsymbol{u}}^{n+1}\right)+2\nu_{r}\left\Vert e_{\boldsymbol{w}}^{n}\right\Vert ^{2}+C\left(\left\Vert \boldsymbol{u}(t^{n})\right\Vert _{2}^{2}+\left\Vert \boldsymbol{u}(t^{n+1})\right\Vert _{2}^{2}+\left\Vert \nabla e_{\boldsymbol{u}}^{n}\right\Vert ^{2}\right)\left\Vert e_{\boldsymbol{u}}^{n}\right\Vert ^{2}\nonumber \\
 & \quad+C\tau\int_{t^{n}}^{t^{n+1}}\left\Vert \boldsymbol{u}_{tt}(s)\right\Vert _{-1}^{2}ds+C\tau\left\Vert \boldsymbol{u}(t^{n+1})\right\Vert _{2}^{2}\int_{t^{n}}^{t^{n+1}}\left\Vert \boldsymbol{u}_{t}(s)\right\Vert ^{2}ds\nonumber \\
 & \quad+C\tau\left\Vert \boldsymbol{u}^{n}\right\Vert \int_{t^{n}}^{t^{n+1}}\left\Vert \boldsymbol{u}_{t}(s)\right\Vert _{2}^{2}ds+C\tau\int_{t^{n}}^{t^{n+1}}\left\Vert \boldsymbol{w}_{t}(s)\right\Vert ^{2}ds,\qquad\forall\ 0\leq n\leq N-1.\label{eq:error_u}
\end{align}
\end{lem}
\begin{proof}
Taking the inner product of \eqref{eq:SAVerroru} with $e_{\boldsymbol{u}}^{n+1}$
and using \eqref{eq:SAVerrordivu}, we obtain 
\begin{align}
 & \frac{\left\Vert e_{\boldsymbol{u}}^{n+1}\right\Vert ^{2}-\left\Vert e_{\boldsymbol{u}}^{n}\right\Vert ^{2}+\left\Vert e_{\boldsymbol{u}}^{n+1}-e_{\boldsymbol{u}}^{n}\right\Vert ^{2}}{2\tau}+\nu_{0}\left\Vert \nabla e_{\boldsymbol{u}}^{n+1}\right\Vert ^{2}\nonumber \\
 & =(R_{\boldsymbol{u}}^{n+1},e_{\boldsymbol{u}}^{n+1})+\left(\boldsymbol{u}(t^{n+1})\cdot\nabla\boldsymbol{u}(t^{n+1})-q^{n+1}\exp\left(\frac{t^{n+1}}{T}\right)\boldsymbol{u}^{n}\cdot\nabla\boldsymbol{u}^{n},e_{\boldsymbol{u}}^{n+1}\right)\nonumber \\
 & \quad+2\nu_{r}\left(\nabla\times\boldsymbol{w}^{n}-\nabla\times\boldsymbol{w}(t^{n+1}),e_{\boldsymbol{u}}^{n+1}\right).\label{eq:e_error_u_inner_product}
\end{align}
For the first term on the right hand side of \eqref{eq:e_error_u_inner_product},
we get 
\begin{equation}
(R_{\boldsymbol{u}}^{n+1},e_{\boldsymbol{u}}^{n+1})\leq\frac{\nu}{8}\left\Vert \nabla e_{\boldsymbol{u}}^{n+1}\right\Vert ^{2}+C\tau\int_{t^{n}}^{t^{n+1}}\left\Vert \boldsymbol{u}_{tt}(s)\right\Vert _{-1}^{2}ds.\label{eq:e_error_inner_Ru}
\end{equation}
For the second term on the right hand side of \eqref{eq:e_error_u_inner_product},
we have 
\begin{align}
 &\left(\boldsymbol{u}(t^{n+1})\cdot\nabla\boldsymbol{u}(t^{n+1})-q^{n+1}\exp(\frac{t^{n+1}}{T})\boldsymbol{u}^{n}\cdot\nabla\boldsymbol{u}^{n},e_{\boldsymbol{u}}^{n+1}\right)\nonumber \\ &=\left((\boldsymbol{u}(t^{n+1})-\boldsymbol{u}^{n})\cdot\nabla\boldsymbol{u}(t^{n+1}),e_{\boldsymbol{u}}^{n+1}\right)+\left(\boldsymbol{u}^{n}\cdot\nabla(\boldsymbol{u}(t^{n+1})-\boldsymbol{u}^{n}),e_{\boldsymbol{u}}^{n+1}\right)\nonumber \\
 & \quad-\exp\left(\frac{t^{n+1}}{T}\right)e_{q}^{n+1}\left(\boldsymbol{u}^{n}\cdot\nabla\boldsymbol{u}^{n},e_{\boldsymbol{u}}^{n+1}\right).\label{eq:e_error_u_nonlinear_convective}
\end{align}
Using Cauchy-Schwarz inequality and \eqref{eq:e_estimate for trilinear form},
the first term on the right hand side of \eqref{eq:e_error_u_nonlinear_convective}
can be bounded by 
\begin{align}
&\left((\boldsymbol{u}(t^{n+1})-\boldsymbol{u}^{n})\cdot\nabla\boldsymbol{u}(t^{n+1}),e_{\boldsymbol{u}}^{n+1}\right)\nonumber \\
&\leq C_{b,1}\left\Vert \boldsymbol{u}(t^{n+1})-\boldsymbol{u}^{n}\right\Vert \left\Vert \boldsymbol{u}(t^{n+1})\right\Vert _{2}\left\Vert \nabla e_{\boldsymbol{u}}^{n+1}\right\Vert \nonumber \\
 & \le\frac{\nu}{8}\left\Vert \nabla e_{\boldsymbol{u}}^{n+1}\right\Vert ^{2}+C\left\Vert \boldsymbol{u}(t^{n+1})-\boldsymbol{u}^{n}\right\Vert ^{2}\left\Vert \boldsymbol{u}(t^{n+1})\right\Vert _{2}^{2}\nonumber \\
 & \le\frac{\nu}{8}\left\Vert \nabla e_{\boldsymbol{u}}^{n+1}\right\Vert ^{2}+C\left\Vert \boldsymbol{u}(t^{n+1})\right\Vert _{2}^{2}\left\Vert e_{\boldsymbol{u}}^{n}\right\Vert ^{2}+C\tau\left\Vert \boldsymbol{u}(t^{n+1})\right\Vert _{2}^{2}\int_{t^{n}}^{t^{n+1}}\left\Vert \boldsymbol{u}_{t}(s)\right\Vert ^{2}ds.\label{eq:e_error_u_nonlinear_convective1}
\end{align}
Similarly, using Cauchy-Schwarz inequality, \eqref{eq:e_estimate for trilinear form}-\eqref{eq:e_estimate for trilinear form2d}
and Young inequality, the second term on the right hand side of \eqref{eq:e_error_u_nonlinear_convective}
can be estimated as follows,
\begin{align}
&(\boldsymbol{u}^{n}\cdot\nabla(\boldsymbol{u}(t^{n+1})-\boldsymbol{u}^{n}),e_{\boldsymbol{u}}^{n+1}) \nonumber \\ &=\left(\boldsymbol{u}^{n}\cdot\nabla(\boldsymbol{u}(t^{n+1})-\boldsymbol{u}(t^{n})),e_{\boldsymbol{u}}^{n+1}\right)-\left(e_{\boldsymbol{u}}^{n}\cdot\nabla e_{\boldsymbol{u}}^{n},e_{\boldsymbol{u}}^{n+1}\right)-\left(\boldsymbol{u}(t^{n})\cdot\nabla e_{\boldsymbol{u}}^{n},e_{\boldsymbol{u}}^{n+1}\right)\nonumber \\
 & \le C_{b,1}\left\Vert \boldsymbol{u}^{n}\right\Vert \left\Vert \boldsymbol{u}(t^{n+1})-\boldsymbol{u}(t^{n})\right\Vert _{2}\left\Vert \nabla e_{\boldsymbol{u}}^{n+1}\right\Vert +C_{b,6}\left\Vert \nabla e_{\boldsymbol{u}}^{n}\right\Vert ^{1/2}\left\Vert e_{\boldsymbol{u}}^{n}\right\Vert ^{1/2}\left\Vert \nabla e_{\boldsymbol{u}}^{n}\right\Vert ^{1/2}\left\Vert e_{\boldsymbol{u}}^{n}\right\Vert ^{1/2}\left\Vert \nabla e_{\boldsymbol{u}}^{n+1}\right\Vert \nonumber \\
 & \quad+C_{b,2}\left\Vert \boldsymbol{u}(t^{n})\right\Vert _{2}\left\Vert e_{\boldsymbol{u}}^{n}\right\Vert \left\Vert \nabla e_{\boldsymbol{u}}^{n+1}\right\Vert \nonumber \\
 & \le C_{b,1}\left\Vert \boldsymbol{u}^{n}\right\Vert \left\Vert \int_{t^{n}}^{t^{n+1}}\boldsymbol{u}_{t}(s)ds\right\Vert _{2}\left\Vert \nabla e_{\boldsymbol{u}}^{n+1}\right\Vert +C_{b,6}\left\Vert \nabla e_{\boldsymbol{u}}^{n}\right\Vert \left\Vert e_{\boldsymbol{u}}^{n}\right\Vert \left\Vert \nabla e_{\boldsymbol{u}}^{n+1}\right\Vert+C_{b,2}\left\Vert \boldsymbol{u}(t^{n})\right\Vert _{2}\left\Vert e_{\boldsymbol{u}}^{n}\right\Vert \left\Vert \nabla e_{\boldsymbol{u}}^{n+1}\right\Vert \nonumber \\
 & \leq\frac{\nu}{8}\left\Vert \nabla e_{\boldsymbol{u}}^{n+1}\right\Vert ^{2}+C\left(\left\Vert \boldsymbol{u}(t^{n})\right\Vert _{2}^{2}+\left\Vert \nabla e_{\boldsymbol{u}}^{n}\right\Vert ^{2}\right)\left\Vert e_{\boldsymbol{u}}^{n}\right\Vert ^{2}+C\tau\left\Vert \boldsymbol{u}^{n}\right\Vert \int_{t^{n}}^{t^{n+1}}\left\Vert \boldsymbol{u}_{t}(s)\right\Vert _{2}^{2}ds.\label{eq:e_error_u_nonlinear_convective2}
\end{align}
For the last term on the right hand side of \eqref{eq:e_error_u_inner_product},
we invoke with Cauchy-Schwarz inequality, Young inequality and \eqref{eq:norm curl}
to get
\begin{align}
&2\nu_{r}\left(\nabla\times\boldsymbol{w}^{n}-\nabla\times\boldsymbol{w}(t^{n+1}),e_{\boldsymbol{u}}^{n+1}\right) \nonumber \\
& =2\nu_{r}\left(\boldsymbol{w}^{n}-\boldsymbol{w}(t^{n+1}),\nabla\times e_{\boldsymbol{u}}^{n+1}\right)\nonumber \\
 & =2\nu_{r}\left(e_{\boldsymbol{w}}^{n},\nabla\times e_{\boldsymbol{u}}^{n+1}\right)-2\nu_{r}\left(\boldsymbol{w}(t^{n+1})-\boldsymbol{w}(t_{n}),\nabla\times e_{\boldsymbol{u}}^{n+1}\right)\nonumber \\
 & \le2\nu_{r}\left\Vert e_{\boldsymbol{w}}^{n}\right\Vert \left\Vert \nabla\times e_{\boldsymbol{u}}^{n+1}\right\Vert +2\nu_{r}\left\Vert \int_{t^{n}}^{t^{n+1}}\boldsymbol{w}_{t}(s)ds\right\Vert \left\Vert \nabla\times e_{\boldsymbol{u}}^{n+1}\right\Vert \nonumber \\
 & \leq\left(\frac{\nu_{r}}{2}+\frac{\nu}{8}\right)\left\Vert \nabla e_{\boldsymbol{u}}^{n+1}\right\Vert ^{2}+2\nu_{r}\left\Vert e_{\boldsymbol{w}}^{n}\right\Vert ^{2}+C\tau\int_{t^{n}}^{t^{n+1}}\left\Vert \boldsymbol{w}_{t}(s)\right\Vert ^{2}ds.\label{eq:e_error_u_couple_force1}
\end{align}
Finally, combining \eqref{eq:e_error_u_inner_product} with \eqref{eq:e_error_inner_Ru}-\eqref{eq:e_error_u_couple_force1}
leads to the desired result. 
\end{proof}
\medskip{}
Next, we derive a estimate for the angular velocity errors in $L^{2}$-
norm.
\begin{lem}
\label{lem: error_estimate_w} Under the assumptions of Theorem \ref{thm: error_estimate_uwq},
we have 
\begin{align}
 & \jmath\begin{aligned}\frac{\left\Vert e_{\boldsymbol{w}}^{n+1}\right\Vert ^{2}-\left\Vert e_{\boldsymbol{w}}^{n}\right\Vert ^{2}+\left\Vert e_{\boldsymbol{w}}^{n+1}-e_{\boldsymbol{w}}^{n}\right\Vert ^{2}}{2\tau}\end{aligned}
+\frac{c_{1}}{2}\left\Vert \nabla e_{\boldsymbol{w}}^{n+1}\right\Vert ^{2}+c_{2}\left\Vert \nabla\cdot e_{\boldsymbol{w}}^{n+1}\right\Vert ^{2}+2\nu_{r}\left\Vert e_{\boldsymbol{w}}^{n+1}\right\Vert ^{2}\nonumber \\
 & \le\frac{\nu_{r}}{2}\left\Vert \nabla e_{\boldsymbol{u}}^{n+1}\right\Vert ^{2}-\jmath\exp\left(\frac{t^{n+1}}{T}\right)e_{q}^{n+1}\left(\boldsymbol{u}^{n}\cdot\nabla\boldsymbol{w}^{n},e_{\boldsymbol{w}}^{n+1}\right)\nonumber \\
 & \quad+C\left(\left\Vert \boldsymbol{w}(t^{n+1})\right\Vert _{2}^{2}+\left\Vert \nabla e_{\boldsymbol{u}}^{n}\right\Vert ^{2}\right)\left\Vert e_{\boldsymbol{u}}^{n}\right\Vert ^{2}+C\left(\left\Vert \boldsymbol{u}(t^{n})\right\Vert _{2}^{2}+\left\Vert \nabla e_{\boldsymbol{w}}^{n}\right\Vert ^{2}\right)\left\Vert e_{\boldsymbol{w}}^{n}\right\Vert ^{2}\nonumber \\
 & \quad+C\tau\int_{t^{n}}^{t^{n+1}}\left\Vert \boldsymbol{w}_{tt}(s)\right\Vert _{-1}^{2}ds+C\tau\left\Vert \boldsymbol{w}(t^{n+1})\right\Vert _{2}^{2}\int_{t^{n}}^{t^{n+1}}\left\Vert \boldsymbol{w}_{t}(s)\right\Vert ^{2}ds\nonumber \\
 & \quad+C\tau\left\Vert \boldsymbol{u}^{n}\right\Vert \int_{t^{n}}^{t^{n+1}}\left\Vert \boldsymbol{w}_{t}(s)\right\Vert _{2}^{2}ds,\qquad\forall\ 0\leq n\leq N-1.\label{eq:errorw}
\end{align}
\end{lem}
\begin{proof}
Taking the inner product of \eqref{eq:SAVerrow} with $e_{\boldsymbol{w}}^{n+1}$,
we obtain 
\begin{align}
 & \begin{aligned}\jmath\frac{\left\Vert e_{\boldsymbol{w}}^{n+1}\right\Vert ^{2}-\left\Vert e_{\boldsymbol{w}}^{n}\right\Vert ^{2}+\left\Vert e_{\boldsymbol{w}}^{n+1}-e_{\boldsymbol{w}}^{n}\right\Vert ^{2}}{2\tau}\end{aligned}
+c_{1}\left\Vert \nabla e_{\boldsymbol{w}}^{n+1}\right\Vert ^{2}+c_{2}\left\Vert \nabla\cdot e_{\boldsymbol{w}}^{n+1}\right\Vert ^{2}+4\nu_{r}\left\Vert e_{\boldsymbol{w}}^{n+1}\right\Vert ^{2}\nonumber \\
 & =(R_{\boldsymbol{w}}^{n+1},e_{\boldsymbol{w}}^{n+1})+\jmath\left((\boldsymbol{u}(t^{n+1})\cdot\nabla)\boldsymbol{w}(t^{n+1})-q^{n+1}\exp\left(\frac{t^{n+1}}{T}\right)(\boldsymbol{u}^{n}\cdot\nabla)\boldsymbol{w}^{n},e_{\boldsymbol{w}}^{n+1}\right)\nonumber \\
 & \quad+\left(2\nu_{r}\nabla\times e_{\boldsymbol{u}}^{n+1},e_{\boldsymbol{w}}^{n+1}\right).\label{eq:e_error_w_inner_product}
\end{align}
For the last term on the right hand side of \eqref{eq:e_error_w_inner_product},
we have 
\begin{equation}
\left(R_{\boldsymbol{w}}^{n+1},e_{\boldsymbol{w}}^{n+1}\right)\leq\frac{c_{1}}{6}\left\Vert \nabla e_{\boldsymbol{w}}^{n+1}\right\Vert ^{2}+C\tau\int_{t^{n}}^{t^{n+1}}\left\Vert \boldsymbol{w}_{tt}(s)\right\Vert _{-1}^{2}ds.\label{eq:e_error_inner_Rw}
\end{equation}
The second term on the right hand side of \eqref{eq:e_error_w_inner_product}
can be estimated as follows by using the similar procedure in \eqref{eq:e_error_u_nonlinear_convective}.
Thus, we recast is as 
\begin{align}
 & \jmath\left((\boldsymbol{u}(t^{n+1})\cdot\nabla)\boldsymbol{w}(t^{n+1})-q^{n+1}\exp\left(\frac{t^{n+1}}{T}\right)(\boldsymbol{u}^{n}\cdot\nabla)\boldsymbol{w}^{n},e_{\boldsymbol{w}}^{n+1}\right)\nonumber \\
 & =\jmath\left((\boldsymbol{u}(t^{n+1})-\boldsymbol{u}^{n})\cdot\nabla\boldsymbol{w}(t^{n+1}),e_{\boldsymbol{w}}^{n+1}\right)+\jmath\left(\boldsymbol{u}^{n}\cdot\nabla(\boldsymbol{w}(t^{n+1})-\boldsymbol{w}^{n}),e_{\boldsymbol{w}}^{n+1}\right)\nonumber \\
 & \quad-\jmath\exp\left(\frac{t^{n+1}}{T}\right)e_{q}^{n+1}\left(\boldsymbol{u}^{n}\cdot\nabla\boldsymbol{w}^{n},e_{\boldsymbol{w}}^{n+1}\right).\label{eq:e_error_w_nonlinear_convective}
\end{align}
For the first term on the right hand side of \eqref{eq:e_error_w_nonlinear_convective},
similar to \eqref{eq:e_error_u_nonlinear_convective1}, we have
\begin{align}
&\jmath\left((\boldsymbol{u}(t^{n+1})-\boldsymbol{u}^{n})\cdot\nabla\boldsymbol{w}(t^{n+1}),e_{\boldsymbol{w}}^{n+1}\right) \nonumber \\
& \leq\jmath C_{b,1}\left\Vert \boldsymbol{u}(t^{n+1})-\boldsymbol{u}^{n}\right\Vert \left\Vert \boldsymbol{w}(t^{n+1})\right\Vert _{2}\left\Vert \nabla e_{\boldsymbol{w}}^{n+1}\right\Vert \nonumber \\
 & \le\frac{c_{1}}{6}\left\Vert \nabla e_{\boldsymbol{w}}^{n+1}\right\Vert ^{2}+C\left\Vert \boldsymbol{u}(t^{n+1})-\boldsymbol{u}^{n}\right\Vert ^{2}\left\Vert \boldsymbol{w}(t^{n+1})\right\Vert _{2}^{2}\nonumber \\
 & \le\frac{c_{1}}{6}\left\Vert \nabla e_{\boldsymbol{w}}^{n+1}\right\Vert ^{2}+C\left\Vert \boldsymbol{w}(t^{n+1})\right\Vert _{2}^{2}\left\Vert e_{\boldsymbol{u}}^{n}\right\Vert ^{2}+C\tau\left\Vert \boldsymbol{w}(t^{n+1})\right\Vert _{2}^{2}\int_{t^{n}}^{t^{n+1}}\left\Vert \boldsymbol{w}_{t}(s)\right\Vert ^{2}ds.\label{eq:e_error_w_nonlinear_convective1}
\end{align}
For the second term on the right hand side of \eqref{eq:e_error_w_nonlinear_convective},
similar to \eqref{eq:e_error_u_nonlinear_convective2}, we get
\begin{align}
&\jmath\left(\boldsymbol{u}^{n}\cdot\nabla(\boldsymbol{w}(t^{n+1})-\boldsymbol{w}^{n}),e_{\boldsymbol{w}}^{n+1}\right)\nonumber \\
 &\le\jmath\left(\boldsymbol{u}^{n}\cdot\nabla(\boldsymbol{w}(t^{n+1})-\boldsymbol{w}(t^{n})),e_{\boldsymbol{w}}^{n+1}\right)-\jmath\left(e_{\boldsymbol{u}}^{n}\cdot\nabla e_{\boldsymbol{w}}^{n},e_{\boldsymbol{w}}^{n+1}\right)-\jmath\left(\boldsymbol{u}(t^{n})\cdot\nabla e_{\boldsymbol{w}}^{n},e_{\boldsymbol{w}}^{n+1}\right)\nonumber \\
 & \le\jmath C_{b,1}\left\Vert \boldsymbol{u}^{n}\right\Vert \left\Vert \boldsymbol{w}(t^{n+1})-\boldsymbol{w}(t^{n})\right\Vert _{2}\left\Vert \nabla e_{\boldsymbol{w}}^{n+1}\right\Vert +\jmath C_{b,6}\left\Vert \nabla e_{\boldsymbol{u}}^{n}\right\Vert ^{1/2}\left\Vert e_{\boldsymbol{u}}^{n}\right\Vert ^{1/2}\left\Vert \nabla e_{\boldsymbol{w}}^{n}\right\Vert ^{1/2}\left\Vert e_{\boldsymbol{w}}^{n}\right\Vert ^{1/2}\left\Vert \nabla e_{\boldsymbol{u}}^{n+1}\right\Vert \nonumber \\
 & \quad+\jmath C_{b,2}\left\Vert \boldsymbol{u}(t^{n})\right\Vert _{2}\left\Vert e_{\boldsymbol{w}}^{n}\right\Vert \left\Vert \nabla e_{\boldsymbol{w}}^{n+1}\right\Vert \nonumber \\
 & \le\jmath C_{b,1}\left\Vert \boldsymbol{u}^{n}\right\Vert \left\Vert \int_{t^{n}}^{t^{n+1}}\boldsymbol{w}_{t}(s)ds\right\Vert _{2}\left\Vert \nabla e_{\boldsymbol{w}}^{n+1}\right\Vert +\jmath C_{b,6}\left\Vert \nabla e_{\boldsymbol{u}}^{n}\right\Vert ^{1/2}\left\Vert e_{\boldsymbol{u}}^{n}\right\Vert ^{1/2}\left\Vert \nabla e_{\boldsymbol{w}}^{n}\right\Vert ^{1/2}\left\Vert e_{\boldsymbol{w}}^{n}\right\Vert ^{1/2}\left\Vert \nabla e_{\boldsymbol{w}}^{n+1}\right\Vert \nonumber \\
 & \quad+\jmath C_{b,2}\left\Vert \boldsymbol{u}(t^{n})\right\Vert _{2}\left\Vert e_{\boldsymbol{w}}^{n}\right\Vert \left\Vert \nabla e_{\boldsymbol{w}}^{n+1}\right\Vert \nonumber \\
 & \le\frac{c_{1}}{6}\left\Vert \nabla e_{\boldsymbol{w}}^{n+1}\right\Vert ^{2}+C\left(\left\Vert \boldsymbol{u}(t^{n})\right\Vert _{2}^{2}+\left\Vert \nabla e_{\boldsymbol{w}}^{n}\right\Vert ^{2}\right)\left\Vert e_{\boldsymbol{w}}^{n}\right\Vert ^{2}+C\left\Vert \nabla e_{\boldsymbol{u}}^{n}\right\Vert ^{2}\left\Vert e_{\boldsymbol{u}}^{n}\right\Vert ^{2}+C\tau\left\Vert \boldsymbol{u}^{n}\right\Vert \int_{t^{n}}^{t^{n+1}}\left\Vert \boldsymbol{w}_{t}(s)\right\Vert _{2}^{2}ds.\label{eq:e_error_w_nonlinear_convective2}
\end{align}
For the last term on the right hand side of \eqref{eq:e_error_w_inner_product},
we obtain
\begin{equation}
\left(2\nu_{r}\nabla\times e_{\boldsymbol{u}}^{n+1},e_{\boldsymbol{w}}^{n+1}\right)\le2\nu_{r}\left\Vert \nabla \times e_{\boldsymbol{u}}^{n+1}\right\Vert \left\Vert e_{\boldsymbol{w}}^{n+1}\right\Vert \le\frac{\nu_{r}}{2}\left\Vert \nabla e_{\boldsymbol{u}}^{n+1}\right\Vert ^{2}+2\nu_{r}\left\Vert e_{\boldsymbol{w}}^{n+1}\right\Vert ^{2}.\label{eq:e_error_w_couple_force1}
\end{equation}
Combining \eqref{eq:e_error_w_inner_product} with \eqref{eq:e_error_inner_Rw}-\eqref{eq:e_error_w_couple_force1}
leads to the desired result. 
\end{proof}
\medskip{}
Now we turn to estimate the error for the auxiliary variable $q$.
\begin{lem}
\label{lem: error_estimate_q} Under the assumptions of Theorem \ref{thm: error_estimate_uwq},
we have 
\begin{align}
 & \frac{|e_{q}^{n+1}|^{2}-|e_{q}^{n}|^{2}+|e_{q}^{n+1}-e_{q}^{n}|^{2}}{2\tau}+\frac{1}{2T}|e_{q}^{n+1}|^{2}\nonumber \\
 & \le\exp\left(\frac{t^{n+1}}{T}\right)e_{q}^{n+1}\left(\boldsymbol{u}^{n}\cdot\nabla\boldsymbol{u}^{n},e_{\boldsymbol{u}}^{n+1}\right)+\jmath\exp\left(\frac{t^{n+1}}{T}\right)e_{q}^{n+1}\left(\boldsymbol{u}^{n}\cdot\nabla\boldsymbol{w}^{n},e_{\boldsymbol{w}}^{n+1}\right)\nonumber \\
 & \quad+\frac{1}{4k_{2}}\left\Vert \nabla\boldsymbol{u}^{n}\right\Vert ^{2}|e_{q}^{n+1}|^{2}+C\left\Vert e_{\boldsymbol{w}}^{n}\right\Vert ^{2}\left\Vert \boldsymbol{w}(t^{n+1})\right\Vert _{2}^{2}\nonumber \\
 & \quad+C\left(\left\Vert \boldsymbol{u}(t^{n+1})\right\Vert _{2}^{2}+\left\Vert \nabla\boldsymbol{u}(t^{n+1})\right\Vert ^{2}\left\Vert \boldsymbol{u}(t^{n+1})\right\Vert _{2}^{2}+\left\Vert \nabla\boldsymbol{w}(t^{n+1})\right\Vert ^{2}\left\Vert \boldsymbol{w}(t^{n+1})\right\Vert _{2}^{2}\right)\left\Vert e_{\boldsymbol{u}}^{n}\right\Vert ^{2}\nonumber \\
 & \quad+C\tau\int_{t^{n}}^{t^{n+1}}\left|q_{tt}(s)\right|^{2}ds+C\tau\left\Vert \boldsymbol{w}(t^{n+1})\right\Vert _{2}^{2}\int_{t^{n}}^{t^{n+1}}\left\Vert \boldsymbol{w}_{t}(s)\right\Vert ^{2}ds\nonumber \\
 & \quad+C\tau\left(\left\Vert \boldsymbol{u}(t^{n+1})\right\Vert _{2}^{2}+\left\Vert \nabla\boldsymbol{u}(t^{n+1})\right\Vert ^{2}\left\Vert \boldsymbol{u}(t^{n+1})\right\Vert _{2}^{2}+\left\Vert \nabla\boldsymbol{w}(t^{n+1})\right\Vert ^{2}\left\Vert \boldsymbol{w}(t^{n+1})\right\Vert _{2}^{2}\right)\int_{t^{n}}^{t^{n+1}}\left\Vert \boldsymbol{u}_{t}(s)\right\Vert ^{2}ds,\ \forall\ 0\leq n\leq N-1,\label{eq:errorq}
\end{align}
where $k_2$ is defined by \eqref{eq:gradu_uniform}.
\end{lem}
\begin{proof}
Multiplying both sides of \eqref{eq:SAVerrorq} by $e_{q}^{n+1}$
gives 
\begin{align}
 & \frac{|e_{q}^{n+1}|^{2}-|e_{q}^{n}|^{2}+|e_{q}^{n+1}-e_{q}^{n}|^{2}}{2\tau}+\frac{1}{T}|e_{q}^{n+1}|^{2}\nonumber \\
 & =R_{q}^{n+1}e_{q}^{n+1}+\exp\left(\frac{t^{n+1}}{T}\right)e_{q}^{n+1}((\boldsymbol{u}^{n}\cdot\nabla\boldsymbol{u}^{n},\boldsymbol{u}^{n+1})-(\boldsymbol{u}(t^{n+1})\cdot\nabla\boldsymbol{u}(t^{n+1}),\boldsymbol{u}(t^{n+1})))\nonumber \\
 & \quad+\jmath\exp\left(\frac{t^{n+1}}{T}\right)e_{q}^{n+1}((\boldsymbol{u}^{n}\cdot\nabla\boldsymbol{w}^{n},\boldsymbol{w}^{n+1})-(\boldsymbol{u}(t^{n+1})\cdot\nabla\boldsymbol{w}(t^{n+1}),\boldsymbol{w}(t^{n+1}))).\label{eq:e_error_q_inner}
\end{align}
The first term on the right hand side of \eqref{eq:e_error_q_inner}
can be estimated as follows
\begin{equation}
\begin{aligned}R\end{aligned}
_{q}^{n+1}e_{q}^{n+1}\leq\frac{1}{6T}|e_{q}^{n+1}|^{2}+C\tau\int_{t^{n}}^{t^{n+1}}\left|q_{tt}(s)\right|^{2}ds.\label{eq:e_error_q_nonlinear4}
\end{equation}
The second term on the right hand side of \eqref{eq:e_error_q_inner}
can be recast as 
\begin{align}
 & \exp\left(\frac{t^{n+1}}{T}\right)e_{q}^{n+1}\left((\boldsymbol{u}^{n}\cdot\nabla\boldsymbol{u}^{n},\boldsymbol{u}^{n+1})-(\boldsymbol{u}(t^{n+1})\cdot\nabla\boldsymbol{u}(t^{n+1}),\boldsymbol{u}(t^{n+1}))\right)\nonumber \\
 & =\exp\left(\frac{t^{n+1}}{T}\right)e_{q}^{n+1}\left(\boldsymbol{u}^{n}\cdot\nabla\boldsymbol{u}^{n},e_{\boldsymbol{u}}^{n+1}\right)+\exp\left(\frac{t^{n+1}}{T}\right)e_{q}^{n+1}\left(\boldsymbol{u}^{n}\cdot\nabla(\boldsymbol{u}^{n}-\boldsymbol{u}(t^{n+1})),\boldsymbol{u}(t^{n+1})\right)\nonumber \\
 & \quad+\exp\left(\frac{t^{n+1}}{T}\right)e_{q}^{n+1}\left((\boldsymbol{u}^{n}-\boldsymbol{u}(t^{n+1}))\cdot\nabla\boldsymbol{u}(t^{n+1}),\boldsymbol{u}(t^{n+1})\right).\label{eq:e_error_q_nonlinear1}
\end{align}
Using \eqref{eq:e_estimate for trilinear form2} and \eqref{eq:gradu_uniform},
the second term on the right hand side of \eqref{eq:e_error_q_nonlinear1}
is bounded by 
\begin{align}
 & \exp\left(\frac{t^{n+1}}{T}\right)e_{q}^{n+1}\left(\boldsymbol{u}^{n}\cdot\nabla(\boldsymbol{u}^{n}-\boldsymbol{u}(t^{n+1})),\boldsymbol{u}(t^{n+1})\right)\nonumber \\
 & \le\exp(1)C_{b,3}|e_{q}^{n+1}|\left\Vert \nabla\boldsymbol{u}^{n}\right\Vert \left\Vert \boldsymbol{u}^{n}-\boldsymbol{u}(t^{n+1})\right\Vert \left\Vert \boldsymbol{u}(t^{n+1})\right\Vert _{2}\nonumber \\
 & \le C\left\Vert \nabla\boldsymbol{u}^{n}\right\Vert \left\Vert \boldsymbol{u}(t^{n+1})-\boldsymbol{u}(t^{n})-e_{\boldsymbol{u}}^{n}\right\Vert \left\Vert \boldsymbol{u}(t^{n+1})\right\Vert _{2}|e_{q}^{n+1}|\nonumber \\
 & \leq\frac{1}{8k_{2}}\left\Vert \nabla\boldsymbol{u}^{n}\right\Vert ^{2}|e_{q}^{n+1}|^{2}+C\left\Vert e_{\boldsymbol{u}}^{n}\right\Vert ^{2}\left\Vert \boldsymbol{u}(t^{n+1})\right\Vert _{2}^{2}+C\tau\left\Vert \boldsymbol{u}(t^{n+1})\right\Vert _{2}^{2}\int_{t^{n}}^{t^{n+1}}\left\Vert \boldsymbol{u}_{t}(s)\right\Vert ^{2}ds,\label{eq:e_error_q_nonlinear2}
\end{align}
where $k_2$ is given by \eqref{eq:gradu_uniform}. In a same manner, the third term on the right hand side of \eqref{eq:e_error_q_nonlinear1}
can be bounded by 
\begin{align}
 & \exp\left(\frac{t^{n+1}}{T}\right)e_{q}^{n+1}\left((\boldsymbol{u}^{n}-\boldsymbol{u}(t^{n+1}))\cdot\nabla\boldsymbol{u}(t^{n+1}),\boldsymbol{u}(t^{n+1})\right)\nonumber \\
 & \leq\exp(1)C_{b,4}\left\Vert \boldsymbol{u}(t^{n+1})-\boldsymbol{u}^{n}\right\Vert \left\Vert \nabla\boldsymbol{u}(t^{n+1})\right\Vert \left\Vert \boldsymbol{u}(t^{n+1})\right\Vert _{2}|e_{q}^{n+1}|\nonumber \\
 & \le C\left\Vert \boldsymbol{u}(t^{n+1})-\boldsymbol{u}(t^{n})-e_{\boldsymbol{u}}^{n}\right\Vert \left\Vert \nabla\boldsymbol{u}(t^{n+1})\right\Vert \left\Vert \boldsymbol{u}(t^{n+1})\right\Vert _{2}|e_{q}^{n+1}|\nonumber \\
 & \leq\frac{1}{6T}|e_{q}^{n+1}|^{2}+C\left\Vert \nabla\boldsymbol{u}(t^{n+1})\right\Vert ^{2}\left\Vert \boldsymbol{u}(t^{n+1})\right\Vert _{2}^{2}\left\Vert e_{\boldsymbol{u}}^{n}\right\Vert ^{2}\nonumber \\
 &\quad+C\tau\left\Vert \nabla\boldsymbol{u}(t^{n+1})\right\Vert ^{2}\left\Vert \boldsymbol{u}(t^{n+1})\right\Vert _{2}^{2}\int_{t^{n}}^{t^{n+1}}\left\Vert \boldsymbol{u}_{t}(s)\right\Vert ^{2}ds.\label{eq:e_error_q_nonlinear3}
\end{align}
Using the similar procedure in \eqref{eq:e_error_q_nonlinear1}, the
last term on the right hand side of \eqref{eq:e_error_q_inner} can
be rewritten as 
\begin{align}
 & \jmath\exp\left(\frac{t^{n+1}}{T}\right)e_{q}^{n+1}\left((\boldsymbol{u}^{n}\cdot\nabla\boldsymbol{w}^{n},\boldsymbol{w}^{n+1})-(\boldsymbol{u}(t^{n+1})\cdot\nabla\boldsymbol{w}(t^{n+1}),\boldsymbol{w}(t^{n+1}))\right)\nonumber \\
 & =\jmath\exp\left(\frac{t^{n+1}}{T}\right)e_{q}^{n+1}\left(\boldsymbol{u}^{n}\cdot\nabla\boldsymbol{w}^{n},e_{\boldsymbol{w}}^{n+1}\right)+\jmath\exp\left(\frac{t^{n+1}}{T}\right)e_{q}^{n+1}\left(\boldsymbol{u}^{n}\cdot\nabla(\boldsymbol{w}^{n}-\boldsymbol{w}(t^{n+1})),\boldsymbol{w}(t^{n+1})\right)\nonumber \\
 & \quad+\jmath\exp\left(\frac{t^{n+1}}{T}\right)e_{q}^{n+1}\left((\boldsymbol{u}^{n}-\boldsymbol{u}(t^{n+1}))\cdot\nabla\boldsymbol{w}(t^{n+1}),\boldsymbol{w}(t^{n+1})\right).\label{eq:e_error_q_nonlinear5}
\end{align}
Similar to \eqref{eq:e_error_q_nonlinear2}, the first term on the
right hand side of \eqref{eq:e_error_q_nonlinear5} can be estimated
by 
\begin{align}
 & \jmath\exp\left(\frac{t^{n+1}}{T}\right)e_{q}^{n+1}\left(\boldsymbol{u}^{n}\cdot\nabla(\boldsymbol{w}^{n}-\boldsymbol{w}(t^{n+1})),\boldsymbol{w}(t^{n+1})\right)\nonumber \\
 & \le\jmath\exp(1)C_{b,3}|e_{q}^{n+1}|\left\Vert \nabla\boldsymbol{u}^{n}\right\Vert \left\Vert \boldsymbol{w}^{n}-\boldsymbol{w}(t^{n+1})\right\Vert \left\Vert \boldsymbol{w}(t^{n+1})\right\Vert _{2}\nonumber \\
 & \le C\left\Vert \nabla\boldsymbol{u}^{n}\right\Vert \left\Vert \boldsymbol{w}(t^{n+1})-\boldsymbol{w}(t^{n})-e_{\boldsymbol{w}}^{n}\right\Vert \left\Vert \boldsymbol{w}(t^{n+1})\right\Vert _{2}|e_{q}^{n+1}|\nonumber \\
 & \leq\frac{1}{8k_{2}}\left\Vert \nabla\boldsymbol{u}^{n}\right\Vert ^{2}|e_{q}^{n+1}|^{2}+C\left\Vert e_{\boldsymbol{w}}^{n}\right\Vert ^{2}\left\Vert \boldsymbol{w}(t^{n+1})\right\Vert _{2}^{2}+C\tau\left\Vert \boldsymbol{w}(t^{n+1})\right\Vert _{2}^{2}\int_{t^{n}}^{t^{n+1}}\left\Vert \boldsymbol{w}_{t}(s)\right\Vert ^{2}ds,\label{eq:e_error_q_nonlinear6}
\end{align}
where $k_2$ is defined by \eqref{eq:gradu_uniform}. For the second term on the right hand side of \eqref{eq:e_error_q_nonlinear5},
similar to \eqref{eq:e_error_q_nonlinear2}, we have 
\begin{align}
 & \jmath\exp\left(\frac{t^{n+1}}{T}\right)e_{q}^{n+1}\left((\boldsymbol{u}^{n}-\boldsymbol{u}(t^{n+1}))\cdot\nabla\boldsymbol{w}(t^{n+1}),\boldsymbol{w}(t^{n+1})\right)\nonumber \\
 & \leq\jmath\exp(1)C_{b,4}\left\Vert \boldsymbol{u}(t^{n+1})-\boldsymbol{u}^{n}\right\Vert \left\Vert \nabla\boldsymbol{w}(t^{n+1})\right\Vert \left\Vert \boldsymbol{w}(t^{n+1})\right\Vert _{2}|e_{q}^{n+1}|\nonumber \\
 & \le C\left\Vert \boldsymbol{u}(t^{n+1})-\boldsymbol{u}(t^{n})-e_{\boldsymbol{u}}^{n}\right\Vert \left\Vert \nabla\boldsymbol{w}(t^{n+1})\right\Vert \left\Vert \boldsymbol{w}(t^{n+1})\right\Vert _{2}|e_{q}^{n+1}|\nonumber \\
 & \leq\frac{1}{6T}|e_{q}^{n+1}|^{2}+C\left\Vert \nabla\boldsymbol{w}(t^{n+1})\right\Vert ^{2}\left\Vert \boldsymbol{w}(t^{n+1})\right\Vert _{2}^{2}\left\Vert e_{\boldsymbol{u}}^{n}\right\Vert ^{2}\nonumber \\
 &\quad+C\tau\left\Vert \nabla\boldsymbol{w}(t^{n+1})\right\Vert ^{2}\left\Vert \boldsymbol{w}(t^{n+1})\right\Vert _{2}^{2}\int_{t^{n}}^{t^{n+1}}\left\Vert \boldsymbol{u}_{t}(s)\right\Vert ^{2}ds.\label{eq:e_error_q_nonlinear7}
\end{align}
Combining \eqref{eq:e_error_q_inner} with \eqref{eq:e_error_q_nonlinear4}-\eqref{eq:e_error_q_nonlinear7}
yields the desired result. 
\end{proof}
Now we are in the position to prove Theorem \ref{thm: error_estimate_uwq}
by using Lemmas \ref{lem: error_estimate_u}-\ref{lem: error_estimate_q}.
\begin{proof}
Summing up \eqref{eq:error_u}, \eqref{eq:errorw}, \eqref{eq:errorq} and using \eqref{eq:u_uniform}, we have 
\begin{align}
 & \frac{\left\Vert e_{\boldsymbol{u}}^{n+1}\right\Vert ^{2}-\left\Vert e_{\boldsymbol{u}}^{n}\right\Vert ^{2}+\left\Vert e_{\boldsymbol{u}}^{n+1}-e_{\boldsymbol{u}}^{n}\right\Vert ^{2}}{2\tau}+\frac{\nu}{2}\left\Vert \nabla e_{\boldsymbol{u}}^{n+1}\right\Vert ^{2}+\jmath\begin{aligned}\frac{\left\Vert e_{\boldsymbol{w}}^{n+1}\right\Vert ^{2}-\left\Vert e_{\boldsymbol{w}}^{n}\right\Vert ^{2}+\left\Vert e_{\boldsymbol{w}}^{n+1}-e_{\boldsymbol{w}}^{n}\right\Vert ^{2}}{2\tau}\end{aligned}
\nonumber \\
 & +\frac{c_{1}}{2}\left\Vert \nabla e_{\boldsymbol{w}}^{n+1}\right\Vert ^{2}+c_{2}\left\Vert \nabla\cdot e_{\boldsymbol{w}}^{n+1}\right\Vert ^{2}+2\nu_{r}\left\Vert e_{\boldsymbol{w}}^{n+1}\right\Vert ^{2}+\frac{|e_{q}^{n+1}|^{2}-|e_{q}^{n}|^{2}}{2\tau}+\frac{|e_{q}^{n+1}-e_{q}^{n}|^{2}}{2\tau}+\frac{1}{2T}|e_{q}^{n+1}|^{2}\nonumber \\
 & \le2\nu_{r}\left\Vert e_{\boldsymbol{w}}^{n}\right\Vert ^{2}+\frac{1}{4k_{2}}\left\Vert \nabla\boldsymbol{u}^{n}\right\Vert ^{2}|e_{q}^{n+1}|^{2}\nonumber \\
 & \quad+C\left(\left\Vert \boldsymbol{u}(t^{n})\right\Vert _{2}^{2}+\left\Vert \boldsymbol{u}(t^{n+1})\right\Vert _{2}^{2}+\left\Vert \boldsymbol{w}(t^{n+1})\right\Vert _{2}^{2}+\left\Vert \nabla e_{\boldsymbol{u}}^{n}\right\Vert ^{2}\right.\nonumber \\
 & \quad\left.+\left\Vert \nabla\boldsymbol{u}(t^{n+1})\right\Vert ^{2}\left\Vert \boldsymbol{u}(t^{n+1})\right\Vert _{2}^{2}+\left\Vert \nabla\boldsymbol{w}(t^{n+1})\right\Vert ^{2}\left\Vert \boldsymbol{w}(t^{n+1})\right\Vert _{2}^{2}\right)\left\Vert e_{\boldsymbol{u}}^{n}\right\Vert ^{2}\nonumber \\
 & \quad+C\left(\left\Vert \boldsymbol{u}(t^{n})\right\Vert _{2}^{2}+\left\Vert \nabla e_{\boldsymbol{w}}^{n}\right\Vert ^{2}+\left\Vert \boldsymbol{w}(t^{n+1})\right\Vert _{2}^{2}\right)\left\Vert e_{\boldsymbol{w}}^{n}\right\Vert ^{2}\nonumber \\
 & \quad+C\tau\int_{t^{n}}^{t^{n+1}}\left\Vert \boldsymbol{u}_{tt}(s)\right\Vert _{-1}^{2}ds+C\tau\int_{t^{n}}^{t^{n+1}}\left\Vert \boldsymbol{w}_{tt}(s)\right\Vert _{-1}^{2}ds+C\tau\left\Vert \boldsymbol{u}^{n}\right\Vert \int_{t^{n}}^{t^{n+1}}\left\Vert \boldsymbol{u}_{t}(s)\right\Vert _{2}^{2}ds\nonumber \\
 & \quad+C\tau\int_{t^{n}}^{t^{n+1}}\left\Vert q_{tt}(s)\right\Vert ^{2}ds+C\tau\left(1+\left\Vert \boldsymbol{w}(t^{n+1})\right\Vert _{2}^{2}\right)\int_{t^{n}}^{t^{n+1}}\left\Vert \boldsymbol{w}_{t}(s)\right\Vert ^{2}ds+C\tau\left\Vert \boldsymbol{u}^{n}\right\Vert \int_{t^{n}}^{t^{n+1}}\left\Vert \boldsymbol{w}_{t}(s)\right\Vert _{2}^{2}ds\nonumber \\
 & \quad+C\tau\left(\left\Vert \boldsymbol{u}(t^{n+1})\right\Vert _{2}^{2}+\left\Vert \nabla\boldsymbol{u}(t^{n+1})\right\Vert ^{2}\left\Vert \boldsymbol{u}(t^{n+1})\right\Vert _{2}^{2}+\left\Vert \nabla\boldsymbol{w}(t^{n+1})\right\Vert ^{2}\left\Vert \boldsymbol{w}(t^{n+1})\right\Vert _{2}^{2}\right)\int_{t^{n}}^{t^{n+1}}\left\Vert \boldsymbol{u}_{t}(s)\right\Vert ^{2}ds.\nonumber \\
 & \le2\nu_{r}\left\Vert e_{\boldsymbol{w}}^{n}\right\Vert ^{2}+\frac{1}{4k_{2}}\left\Vert \nabla\boldsymbol{u}^{n}\right\Vert ^{2}|e_{q}^{n+1}|^{2}+C\left(1+\left\Vert \nabla e_{\boldsymbol{u}}^{n}\right\Vert ^{2}\right)\left\Vert e_{\boldsymbol{u}}^{n}\right\Vert ^{2}+C\tau\left(1+\left\Vert \nabla e_{\boldsymbol{w}}^{n}\right\Vert ^{2}\right)\left\Vert e_{\boldsymbol{w}}^{n}\right\Vert ^{2}\nonumber \\
 & \quad+C\tau\int_{t^{n}}^{t^{n+1}}\left(\left\Vert \boldsymbol{u}_{tt}(s)\right\Vert _{-1}^{2}+\left\Vert \boldsymbol{w}_{tt}(s)\right\Vert _{-1}^{2}+\left\Vert \boldsymbol{u}_{t}(s)\right\Vert _{2}^{2}+\left|q_{tt}(s)\right|^{2}+\left\Vert \boldsymbol{w}_{t}(s)\right\Vert_2 ^{2}\right)ds,\label{eq:e_error_uwq_final1}
\end{align}
We first deduce a bound for $\left|e_{q}^{m^{*}+1}\right|$, where
$m^{*}$ is the time step such that 
\begin{equation}
\left|e_{q}^{m^{*}+1}\right|=\max_{0\leq n\leq N-1}\left|e_{q}^{n+1}\right|.\label{eq:eqmax}
\end{equation}
Multiplying \eqref{eq:e_error_uwq_final1} by $2\tau$, summing up
over $n$ from 0 to $m^{*}$ and using \eqref{eq:eqmax}, we get
\begin{align}
 & \left\Vert e_{\boldsymbol{u}}^{m^{*}+1}\right\Vert ^{2}+\nu\tau\sum_{n=0}^{m^{*}}\left\Vert \nabla e_{\boldsymbol{u}}^{n+1}\right\Vert ^{2}+\left(\jmath+4\nu\tau\right)\left\Vert e_{\boldsymbol{w}}^{m^{*}+1}\right\Vert ^{2}+c_{1}\tau\sum_{n=0}^{m^{*}}\left\Vert \nabla e_{\boldsymbol{w}}^{n+1}\right\Vert ^{2}\nonumber \\
 & +2c_{2}\tau\sum_{n=0}^{m^{*}}\left\Vert \nabla\cdot e_{\boldsymbol{w}}^{n+1}\right\Vert ^{2}+|e_{q}^{m^{*}+1}|^{2}+\frac{\tau}{T}\sum_{n=0}^{m^{*}}|e_{q}^{n+1}|^{2}\nonumber \\
 & \le\frac{\tau}{2k_{2}}\sum_{n=0}^{m^{*}}\left\Vert \nabla\boldsymbol{u}^{n}\right\Vert ^{2}|e_{q}^{m^{*}+1}|^{2}+C\tau\sum_{n=0}^{m^{*}}\left(1+\left\Vert \nabla e_{\boldsymbol{u}}^{n}\right\Vert ^{2}\right)\left\Vert e_{\boldsymbol{u}}^{n}\right\Vert ^{2}+C\tau\sum_{n=0}^{m^{*}}\left(1+\left\Vert \nabla e_{\boldsymbol{w}}^{n}\right\Vert ^{2}\right)\left\Vert e_{\boldsymbol{w}}^{n}\right\Vert ^{2}\nonumber \\
 & \quad+C\tau^{2}\int_{0}^{t^{m^{*}+1}}\left(\left\Vert \boldsymbol{u}_{tt}(s)\right\Vert _{-1}^{2}+\left\Vert \boldsymbol{w}_{tt}(s)\right\Vert _{-1}^{2}+\left\Vert \boldsymbol{u}_{t}(s)\right\Vert _{2}^{2}+\left|q_{tt}(s)\right|^{2}+\left\Vert \boldsymbol{w}_{t}(s)\right\Vert ^{2}\right)ds.\label{eq:e_error_uwq_final2}
\end{align}
Using \eqref{eq:gradu_uniform}, we get 
\begin{align*}
\frac{\tau}{2k_{2}}\sum_{n=0}^{m^{*}}\left\Vert \nabla\boldsymbol{u}^{n}\right\Vert ^{2}|e_{q}^{m^{*}+1}|^{2}  \le\frac{1}{2}\left|e_{q}^{m^{*}+1}\right|^{2},\\
\tau\sum_{n=0}^{m^{*}}\left(1+\left\Vert \nabla e_{\boldsymbol{u}}^{n}\right\Vert ^{2}\right)\le C,\quad \tau\sum_{n=0}^{m^{*}}\left(1+\left\Vert \nabla e_{\boldsymbol{w}}^{n}\right\Vert ^{2}\right)\le C.
\end{align*}
Invoking with the discrete Gronwall inequality in Lemma \ref{lem: gronwall2},
we obtain
\begin{align}
 & \left\Vert e_{\boldsymbol{u}}^{m^{*}+1}\right\Vert ^{2}+\nu\tau\sum_{n=0}^{m^{*}}\left\Vert \nabla e_{\boldsymbol{u}}^{n+1}\right\Vert ^{2}+\left(\jmath+4\nu\tau\right)\left\Vert e_{\boldsymbol{w}}^{m^{*}+1}\right\Vert ^{2}+c_{1}\tau\sum_{n=0}^{m^{*}}\left\Vert \nabla e_{\boldsymbol{w}}^{n+1}\right\Vert ^{2}\nonumber \\
 & +2c_{2}\tau\sum_{n=0}^{m^{*}}\left\Vert \nabla\cdot e_{\boldsymbol{w}}^{n+1}\right\Vert ^{2}+\frac{1}{2}|e_{q}^{m^{*}+1}|^{2}+\frac{\tau}{T}\sum_{n=0}^{m^{*}}|e_{q}^{n+1}|^{2}\nonumber \\
 & \le C\tau^{2}\int_{0}^{t^{m^{*}+1}}\left(\left\Vert \boldsymbol{u}_{tt}(s)\right\Vert _{-1}^{2}+\left\Vert \boldsymbol{w}_{tt}(s)\right\Vert _{-1}^{2}+\left\Vert \boldsymbol{u}_{t}(s)\right\Vert _{2}^{2}+\left|q_{tt}(s)\right|^{2}+\left\Vert \boldsymbol{w}_{t}(s)\right\Vert ^{2}\right)ds.\label{eq:e_error_uwq_final3}
\end{align}
Now turning to \eqref{eq:e_error_uwq_final1}, multiply it by $2\tau$
and sum up over $n$ from 0 to $m$, and using \eqref{eq:eqmax}, we derive
\begin{align}
 & \left\Vert e_{\boldsymbol{u}}^{m+1}\right\Vert ^{2}+\nu\tau\sum_{n=0}^{m}\left\Vert \nabla e_{\boldsymbol{u}}^{n+1}\right\Vert ^{2}+\left(\jmath+4\nu\tau\right)\left\Vert e_{\boldsymbol{w}}^{m+1}\right\Vert ^{2}+c_{1}\tau\sum_{n=0}^{m}\left\Vert \nabla e_{\boldsymbol{w}}^{n+1}\right\Vert ^{2}\nonumber \\
 & +2c_{2}\tau\sum_{n=0}^{m}\left\Vert \nabla\cdot e_{\boldsymbol{w}}^{n+1}\right\Vert ^{2}+|e_{q}^{m+1}|^{2}+\frac{\tau}{T}\sum_{n=0}^{m}|e_{q}^{n+1}|^{2}\nonumber \\
 & \le\frac{\tau}{2k_{2}}\sum_{n=0}^{m}\left\Vert \nabla\boldsymbol{u}^{n}\right\Vert ^{2}|e_{q}^{m^{*}+1}|^{2}+C\tau\sum_{n=0}^{m}\left(1+\left\Vert \nabla e_{\boldsymbol{u}}^{n}\right\Vert ^{2}\right)\left\Vert e_{\boldsymbol{u}}^{n}\right\Vert ^{2}+C\tau\sum_{n=0}^{m}\left(1+\left\Vert \nabla e_{\boldsymbol{w}}^{n}\right\Vert ^{2}\right)\left\Vert e_{\boldsymbol{w}}^{n}\right\Vert ^{2}\nonumber \\
 & \quad+C\tau^{2}\int_{0}^{t^{m+1}}\left(\left\Vert \boldsymbol{u}_{tt}(s)\right\Vert _{-1}^{2}+\left\Vert \boldsymbol{w}_{tt}(s)\right\Vert _{-1}^{2}+\left\Vert \boldsymbol{u}_{t}(s)\right\Vert _{2}^{2}+\left|q_{tt}(s)\right|^{2}+\left\Vert \boldsymbol{w}_{t}(s)\right\Vert ^{2}\right)ds.\label{eq:e_error_uwq_final4}
\end{align}
We deduce form \eqref{eq:gradu_uniform} again that
\[
\tau\sum_{n=0}^{m}\left(1+\left\Vert \nabla e_{\boldsymbol{u}}^{n}\right\Vert ^{2}\right)\le C,\quad\tau\sum_{n=0}^{m}\left(1+\left\Vert \nabla e_{\boldsymbol{w}}^{n}\right\Vert ^{2}\right)\le C.
\]
Using \eqref{eq:e_error_uwq_final3} and the discrete Gronwall inequality
in Lemma \ref{lem: gronwall2} to \eqref{eq:e_error_uwq_final4},
we complete the proof.
\end{proof}

\subsection{Error estimates for the pressure}

The section is devoted to prove the error estimate for the pressure.
For this end, we first establish the estimate on $\delta_{t}e_{\boldsymbol{u}}^{n+1}$.
\begin{lem}
\label{lem:est_du}Assume the exact solution satisfies $\boldsymbol{u}\in H^{2}(0,T;\boldsymbol{H}^{2}(\Omega))\bigcap H^{1}(0,T;\boldsymbol{H}^{2}(\Omega))\bigcap L^{\infty}(0,T;\boldsymbol{H}^{2}(\Omega))$ and $\boldsymbol{w}\in H^{2}(0,T;\boldsymbol{H}^{-1}(\Omega))\bigcap H^{1}(0,T;\boldsymbol{H}^{2}(\Omega))\bigcap L^{\infty}(0,T;\boldsymbol{H}^{2}(\Omega))$,
then we have the following error estimate for $m\ge0$
\begin{equation}
\left\Vert \nabla e_{\boldsymbol{u}}^{m+1}\right\Vert ^{2}+\tau\sum\limits _{n=0}^{m}\left\Vert \delta_{t}e_{\boldsymbol{u}}^{n+1}\right\Vert ^{2}+\nu_0\tau\sum\limits _{n=0}^{m}\left\Vert Ae_{\boldsymbol{u}}^{n+1}\right\Vert ^{2}\le C\tau^{2}.\label{eq:est_du}
\end{equation}
\end{lem}
\begin{proof}
First of all, in virtue of \eqref{eq:estuwq}, we have
\[
\left\Vert \nabla e_{\boldsymbol{u}}^{n+1}\right\Vert ^{2}\le\tau^{-1}\left(\tau\sum_{n=0}^{m}\left\Vert \nabla e_{\boldsymbol{u}}^{n+1}\right\Vert ^{2}\right)\leq C\tau,\quad\left\Vert \nabla e_{\boldsymbol{w}}^{n+1}\right\Vert ^{2}\le\tau^{-1}\left(\tau\sum_{n=0}^{m}\left\Vert \nabla e_{\boldsymbol{w}}^{n+1}\right\Vert ^{2}\right)\leq C\tau.
\]
Hence, there holds that
\begin{equation}
\left\Vert \nabla\boldsymbol{u}^{n+1}\right\Vert \leq\left\Vert \nabla e_{\boldsymbol{u}}^{n+1}\right\Vert +\left\Vert \nabla\boldsymbol{u}\left(t^{n+1}\right)\right\Vert \leq C\left(\tau^{1/2}+\left\Vert \nabla\boldsymbol{u}\left(t^{n+1}\right)\right\Vert \right).\label{eq:gradu}
\end{equation}
Taking the inner product of \eqref{eq:SAVerroru} with $Ae_{\boldsymbol{u}}^{n+1}+\delta_{t}e_{\boldsymbol{u}}^{n+1}$,
we obtain 
\begin{align}
 & (1+\nu_0)\frac{\left\Vert \nabla e_{\boldsymbol{u}}^{n+1}\right\Vert ^{2}-\left\Vert \nabla e_{\boldsymbol{u}}^{n}\right\Vert ^{2}+\left\Vert \nabla e_{\boldsymbol{u}}^{n+1}-\nabla e_{\boldsymbol{u}}^{n}\right\Vert ^{2}}{2\tau}+\left\Vert \delta_{t}e_{\boldsymbol{u}}^{n+1}\right\Vert ^{2}+\nu_0\left\Vert Ae_{\boldsymbol{u}}^{n+1}\right\Vert ^{2}\nonumber \\
 & =\left(R_{\boldsymbol{u}}^{n+1},Ae_{\boldsymbol{u}}^{n+1}+\delta_{t}e_{\boldsymbol{u}}^{n+1}\right)+\left(2\nu_{r}\left(\nabla\times\boldsymbol{w}^{n}-\nabla\times\boldsymbol{w}(t^{n+1})\right),Ae_{\boldsymbol{u}}^{n+1}+\delta_{t}e_{\boldsymbol{u}}^{n+1}\right)\nonumber \\
 & +\left(\boldsymbol{u}\left(t^{n+1}\right)\cdot\nabla\boldsymbol{u}\left(t^{n+1}\right)-q^{n+1}\exp{\left(\frac{t^{n+1}}{T}\right)}\boldsymbol{u}^{n}\cdot\nabla\boldsymbol{u}^{n},Ae_{\boldsymbol{u}}^{n+1}+\delta_{t}e_{\boldsymbol{u}}^{n+1}\right).\label{eq:e_error_Au}
\end{align}
For the first term on the right hand side of \eqref{eq:e_error_Au},
we get 
\begin{equation}
\left(R_{\boldsymbol{u}}^{n+1},Ae_{\boldsymbol{u}}^{n+1}+\delta_{t}e_{\boldsymbol{u}}^{n+1}\right)\leq\frac{1}{12}\left\Vert \delta_{t}e_{\boldsymbol{u}}^{n+1}\right\Vert ^{2}+\frac{\nu_0}{24}\left\Vert Ae_{\boldsymbol{u}}^{n+1}\right\Vert ^{2}+C\tau\int_{t^{n}}^{t^{n+1}}\left\Vert \boldsymbol{u}_{tt}(s)\right\Vert ^{2}ds.\label{eq:e_error_Au_Ru}
\end{equation}
For the second term on the right hand side of \eqref{eq:e_error_Au},
we get 
\begin{align*}
&\left(2\nu_{r}\left(\nabla\times\boldsymbol{w}^{n}-\nabla\times\boldsymbol{w}(t^{n+1})\right),Ae_{\boldsymbol{u}}^{n+1}+\delta_{t}e_{\boldsymbol{u}}^{n+1}\right) \\
& =2\nu_{r}\left(\nabla\times e_{\boldsymbol{w}}^{n},Ae_{\boldsymbol{u}}^{n+1}+\delta_{t}e_{\boldsymbol{u}}^{n+1}\right)-2\nu_{r}\left(\nabla\times\left(\boldsymbol{w}(t^{n+1})-\boldsymbol{w}(t_{n})\right),Ae_{\boldsymbol{u}}^{n+1}+\delta_{t}e_{\boldsymbol{u}}^{n+1}\right)\\
 & \leq\frac{1}{6}\left\Vert \delta_{t}e_{\boldsymbol{u}}^{n+1}\right\Vert ^{2}+\frac{\nu_0}{12}\left\Vert Ae_{\boldsymbol{u}}^{n+1}\right\Vert ^{2}+C\left\Vert \nabla e_{\boldsymbol{w}}^{n}\right\Vert ^{2}+C\tau\int_{t^{n}}^{t^{n+1}}\left\Vert \nabla\boldsymbol{w}_{t}(s)\right\Vert ^{2}ds
\end{align*}
For the last term on the right hand side of \eqref{eq:e_error_Au},
we have 
\begin{align}
 & \left(\boldsymbol{u}(t^{n+1})\cdot\nabla\boldsymbol{u}(t^{n+1})-q^{n+1}\exp{\left(\frac{t^{n+1}}{T}\right)}\boldsymbol{u}^{n}\cdot\nabla\boldsymbol{u}^{n},Ae_{\boldsymbol{u}}^{n+1}+\delta_{t}e_{\boldsymbol{u}}^{n+1}\right)\nonumber \\
 & =-e_{q}^{n+1}\exp{\left(\frac{t^{n+1}}{T}\right)}\left((\boldsymbol{u}^{n}\cdot\nabla)\boldsymbol{u}^{n},Ae_{\boldsymbol{u}}^{n+1}+\delta_{t}e_{\boldsymbol{u}}^{n+1}\right) +\left((\boldsymbol{u}(t^{n+1})-\boldsymbol{u}^{n})\cdot\nabla\boldsymbol{u}(t^{n+1}),Ae_{\boldsymbol{u}}^{n+1}+\delta_{t}e_{\boldsymbol{u}}^{n+1}\right)\nonumber \\
 & \quad+\left(\boldsymbol{u}^{n}\cdot\nabla(\boldsymbol{u}(t^{n+1})-\boldsymbol{u}^{n}),Ae_{\boldsymbol{u}}^{n+1}+\delta_{t}e_{\boldsymbol{u}}^{n+1}\right).\label{eq:e_error_Au_nonlinear1}
\end{align}
Using \eqref{eq:e_estimate for trilinear form2d1}, \eqref{eq:e_estimate for trilinear form4}
and \eqref{eq:gradu}, the first term on the right hand side of \eqref{eq:e_error_Au_nonlinear1}
can be bounded by 
\begin{align}
 & \begin{aligned}-\end{aligned}
e_{q}^{n+1}\exp{\left(\frac{t^{n+1}}{T}\right)}\left(\boldsymbol{u}^{n}\cdot\nabla\boldsymbol{u}^{n},Ae_{\boldsymbol{u}}^{n+1}+\delta_{t}e_{\boldsymbol{u}}^{n+1}\right)\nonumber \\
 & =-e_{q}^{n+1}\exp{\left(\frac{t^{n+1}}{T}\right)}\left(\boldsymbol{u}^{n}\cdot\nabla e_{\boldsymbol{u}}^{n},Ae_{\boldsymbol{u}}^{n+1}+\delta_{t}e_{\boldsymbol{u}}^{n+1}\right) -e_{q}^{n+1}\exp{\left(\frac{t^{n+1}}{T}\right)}\left((\boldsymbol{u}^{n}\cdot\nabla\boldsymbol{u}(t^{n}),Ae_{\boldsymbol{u}}^{n+1}+\delta_{t}e_{\boldsymbol{u}}^{n+1}\right)\nonumber \\
 & \leq C_{b,7}|e_{q}^{n+1}|\left\Vert \boldsymbol{u}^{n}\right\Vert ^{1/2}\left\Vert \nabla\boldsymbol{u}^{n}\right\Vert ^{1/2}\left\Vert \nabla e_{\boldsymbol{u}}^{n}\right\Vert ^{1/2}\left\Vert Ae_{\boldsymbol{u}}^{n}\right\Vert ^{1/2}\left\Vert Ae_{\boldsymbol{u}}^{n+1}+\delta_{t}e_{\boldsymbol{u}}^{n+1}\right\Vert \nonumber \\
 & \quad+C_{b,5}|e_{q}^{n+1}|\left\Vert \nabla\boldsymbol{u}^{n}\right\Vert \left\Vert \boldsymbol{u}(t^{n})\right\Vert _{2}\left\Vert Ae_{\boldsymbol{u}}^{n+1}+\delta_{t}e_{\boldsymbol{u}}^{n+1}\right\Vert \nonumber \\
 & \leq\frac{1}{12}\left\Vert \delta_{t}e_{\boldsymbol{u}}^{n+1}\right\Vert ^{2}+\frac{\nu_0}{24}\left\Vert Ae_{\boldsymbol{u}}^{n+1}\right\Vert ^{2}+\frac{\nu_0}{8}\left\Vert Ae_{\boldsymbol{u}}^{n}\right\Vert ^{2}\nonumber \\
 & \quad+C\left(\tau+\left\Vert \nabla\boldsymbol{u}(t^{n})\right\Vert ^{2}\right)\left\Vert \nabla e_{\boldsymbol{u}}^{n}\right\Vert ^{2}+C\left(\tau+\left\Vert \nabla\boldsymbol{u}(t^{n})\right\Vert ^{2}\right)\left\Vert \boldsymbol{u}(t^{n})\right\Vert _{2}|e_{q}^{n+1}|^{2}.\label{eq:e_error_Au_nonlinear2}
\end{align}
Similarly, the second term on the right hand side of \eqref{eq:e_error_Au_nonlinear1}
can be estimated by 
\begin{align}
 & \begin{aligned}\big(\end{aligned}
(\boldsymbol{u}(t^{n+1})-\boldsymbol{u}^{n})\cdot\nabla\boldsymbol{u}(t^{n+1}),Ae_{\boldsymbol{u}}^{n+1}+\delta_{t}e_{\boldsymbol{u}}^{n+1}\big)\nonumber \\
 & \leq C_{b,5}\left\Vert \nabla\boldsymbol{u}(t^{n+1})-\nabla\boldsymbol{u}^{n}\right\Vert \left\Vert \boldsymbol{u}(t^{n+1})\right\Vert _{2}\left\Vert Ae_{\boldsymbol{u}}^{n+1}+\delta_{t}e_{\boldsymbol{u}}^{n+1}\right\Vert \nonumber \\
 & \leq\frac{1}{12}\left\Vert \delta_{t}e_{\boldsymbol{u}}^{n+1}\right\Vert ^{2}+\frac{\nu_0}{24}\left\Vert Ae_{\boldsymbol{u}}^{n+1}\right\Vert ^{2}+C\left\Vert \boldsymbol{u}(t^{n+1})\right\Vert _{2}\left\Vert \nabla e_{\boldsymbol{u}}^{n}\right\Vert ^{2}+C\left\Vert \boldsymbol{u}(t^{n+1})\right\Vert _{2}^{2}\tau\int_{t^{n}}^{t^{n+1}}\left\Vert \nabla\boldsymbol{u}_{t}(s)\right\Vert ^{2}ds. \label{eq:e_error_Au_nonlinear3}
\end{align}
The last term on the right hand side of \eqref{eq:e_error_Au_nonlinear1}
can be bounded by 
\begin{align}
 & \begin{aligned}\big(\end{aligned}
\boldsymbol{u}^{n}\cdot\nabla(\boldsymbol{u}(t^{n+1})-\boldsymbol{u}^{n}),Ae_{\boldsymbol{u}}^{n+1}+\delta_{t}e_{\boldsymbol{u}}^{n+1}\big)\nonumber \\
 & =\left(\boldsymbol{u}^{n}\cdot\nabla(\boldsymbol{u}(t^{n+1})-\boldsymbol{u}(t^{n})),Ae_{\boldsymbol{u}}^{n+1}+\delta_{t}e_{\boldsymbol{u}}^{n+1}\right)-\left(\boldsymbol{u}^{n}\cdot\nabla e_{\boldsymbol{u}}^{n},Ae_{\boldsymbol{u}}^{n+1}+\delta_{t}e_{\boldsymbol{u}}^{n+1}\right)\nonumber \\
 & \leq C_{b,5}\left\Vert \nabla\boldsymbol{u}^{n}\right\Vert \left\Vert \boldsymbol{u}(t^{n+1})-\boldsymbol{u}(t^{n})\right\Vert _{2}\left\Vert Ae_{\boldsymbol{u}}^{n+1}+\delta_{t}e_{\boldsymbol{u}}^{n+1}\right\Vert \nonumber \\
 &+C_{b,7}\left\Vert \boldsymbol{u}^{n}\right\Vert ^{1/2}\left\Vert \nabla\boldsymbol{u}^{n}\right\Vert ^{1/2}\left\Vert \nabla e_{\boldsymbol{u}}^{n}\right\Vert ^{1/2}\left\Vert Ae_{\boldsymbol{u}}^{n}\right\Vert ^{1/2}\left\Vert Ae_{\boldsymbol{u}}^{n+1}+\delta_{t}e_{\boldsymbol{u}}^{n+1}\right\Vert \nonumber \\
 & \leq\frac{1}{12}\left\Vert \delta_{t}e_{\boldsymbol{u}}^{n+1}\right\Vert ^{2}+\frac{\nu_0}{24}\left\Vert Ae_{\boldsymbol{u}}^{n+1}\right\Vert ^{2}+C\left(\tau+\left\Vert \nabla\boldsymbol{u}(t^{n})\right\Vert ^{2}\right)\left\Vert \nabla e_{\boldsymbol{u}}^{n}\right\Vert ^{2}\nonumber \\
 & \quad+\frac{\nu_0}{8}\left\Vert Ae_{\boldsymbol{u}}^{n}\right\Vert ^{2}+C\tau\left(\tau+\left\Vert \nabla\boldsymbol{u}(t^{n})\right\Vert ^{2}\right)\int_{t^{n}}^{t^{n+1}}\left\Vert \boldsymbol{u}_{t}(s)\right\Vert _{2}^{2}ds.\label{eq:e_error_Au_nonlinear4}
\end{align}
Combining \eqref{eq:e_error_Au} with \eqref{eq:e_error_Au_Ru}-\eqref{eq:e_error_Au_nonlinear4},
we have 
\begin{align}
 & (1+\nu_0)\frac{\left\Vert \nabla e_{\boldsymbol{u}}^{n+1}\right\Vert ^{2}-\left\Vert \nabla e_{\boldsymbol{u}}^{n}\right\Vert ^{2}+\left\Vert \nabla e_{\boldsymbol{u}}^{n+1}-\nabla e_{\boldsymbol{u}}^{n}\right\Vert ^{2}}{2\tau}+\frac{1}{2}\left\Vert \delta_{t}e_{\boldsymbol{u}}^{n+1}\right\Vert ^{2}+\frac{3\nu_0}{4}\left\Vert Ae_{\boldsymbol{u}}^{n+1}\right\Vert ^{2}\nonumber \\
 & \le\frac{\nu_0}{4}\left\Vert Ae_{\boldsymbol{u}}^{n}\right\Vert ^{2}+C\left(1+\tau+\left\Vert \nabla\boldsymbol{u}(t^{n})\right\Vert ^{2}\right)\left(\left\Vert \nabla e_{\boldsymbol{u}}^{n}\right\Vert ^{2}+|e_{q}^{n+1}|^{2}\right)+C\left\Vert \nabla e_{\boldsymbol{w}}^{n}\right\Vert ^{2}\nonumber \\
 & \quad+C\tau\int_{t^{n}}^{t^{n+1}}\left(\left\Vert \boldsymbol{u}_{t}(s)\right\Vert _{2}^{2}ds+\left\Vert \nabla\boldsymbol{u}_{t}(s)\right\Vert ^{2}+\left\Vert \boldsymbol{u}_{tt}(s)\right\Vert ^{2}+\left\Vert \nabla\boldsymbol{w}_{t}(s)\right\Vert ^{2}\right)ds.\label{eq:e_error_AuAb_final1}
\end{align}
Multiplying \eqref{eq:e_error_AuAb_final1} by $2\tau$ and summing
over $n$ from 0 to $m$, and applying the discrete Gronwall inequality
in Lemma \ref{lem: gronwall2}, we obtain 
\begin{align}
 & \left\Vert \nabla e_{\boldsymbol{u}}^{m+1}\right\Vert ^{2}+\tau\sum\limits _{n=0}^{m}\left\Vert \delta_{t}e_{\boldsymbol{u}}^{n+1}\right\Vert ^{2}+\nu_0\tau\sum\limits _{n=0}^{m}\left\Vert Ae_{\boldsymbol{u}}^{n+1}\right\Vert ^{2}\nonumber \\
 & \le C\left(1+\tau+\left\Vert \nabla\boldsymbol{u}(t^{n})\right\Vert ^{2}\right)\tau\sum\limits _{n=0}^{m}\left(\left\Vert \nabla e_{\boldsymbol{u}}^{n}\right\Vert ^{2}+|e_{q}^{n+1}|^{2}\right)+C\tau^{2}.\label{eq:e_error_AuAb_final2}
\end{align}
Combining the above estimate with Theorem \ref{thm: error_estimate_uwq},
we obtain the desired result.
\end{proof}
\medskip{}

We are now in position to prove the error estimates for the pressure. 
\begin{thm}
\label{thm:error_estimate_p}Assume the exact solution satisfies $\boldsymbol{u}\in H^{2}(0,T;\boldsymbol{H}^{2}(\Omega))\bigcap H^{1}(0,T;\boldsymbol{H}^{2}(\Omega))\bigcap L^{\infty}(0,T;\boldsymbol{H}^{2}(\Omega))$
,$p\in L^{2}(0,T;L_{0}^{2}(\Omega))$ and $\boldsymbol{w}\in H^{2}(0,T;\boldsymbol{H}^{-1}(\Omega))\bigcap H^{1}(0,T;\boldsymbol{H}^{2}(\Omega))\bigcap L^{\infty}(0,T;\boldsymbol{H}^{2}(\Omega))$, then we have the following
error estimate for $m\ge0$
\begin{equation}
\tau\sum_{n=0}^{m}\left\Vert e_{p}^{n+1}\right\Vert ^{2}\le C\tau^{2}.\label{eq:est_p}
\end{equation}
\end{thm}
\begin{proof}
Taking the inner product of \eqref{eq:SAVerroru} with $\boldsymbol{v}\in\boldsymbol{X}$,
we obtain 
\begin{align}
\begin{aligned}(\end{aligned}
\nabla e_{p}^{n+1},\boldsymbol{v}) & =-(\delta_{t}e_{\boldsymbol{u}}^{n+1},\boldsymbol{v})-\nu_{0}(\nabla e_{\boldsymbol{u}}^{n+1},\nabla\boldsymbol{v})+(R_{\boldsymbol{u}}^{n+1},\boldsymbol{v})+2\nu_{r}\left(\nabla\times\boldsymbol{w}^{n}-\nabla\times\boldsymbol{w}(t^{n+1}),\boldsymbol{v}\right)\nonumber \\
 & \quad+\left((\boldsymbol{u}(t^{n+1})\cdot\nabla)\boldsymbol{u}(t^{n+1})-q^{n+1}\exp{\left(\frac{t^{n+1}}{T}\right)}(\boldsymbol{u}^{n}\cdot\nabla)\boldsymbol{u}^{n},\boldsymbol{v}\right).\label{eq:e_error_p1}
\end{align}
Note that we only need to estimate the last two term on the right hand
side of \eqref{eq:e_error_p1}. For all $\boldsymbol{v}\in\boldsymbol{X}$, we use Cauchy-Schwarz inequality, Young inequality and \eqref{eq:norm curl}
to estimate the fourth term as
\begin{align}
2\nu_{r}\left(\nabla\times\boldsymbol{w}^{n}-\nabla\times\boldsymbol{w}(t^{n+1}),\boldsymbol{v}\right)&=2\nu_{r}\left(\boldsymbol{w}^{n}-\boldsymbol{w}(t^{n+1}),\nabla\times \boldsymbol{v}\right)\nonumber \\
	& =2\nu_{r}\left(e_{\boldsymbol{w}}^{n},\nabla\times \boldsymbol{v}\right)-2\nu_{r}\left(\boldsymbol{w}(t^{n+1})-\boldsymbol{w}(t_{n}),\nabla\times \boldsymbol{v}\right)\nonumber \\
	& \le2\nu_{r}\left\Vert e_{\boldsymbol{w}}^{n}\right\Vert \left\Vert \nabla \boldsymbol{v}\right\Vert +2\nu_{r}\left\Vert \int_{t^{n}}^{t^{n+1}}\boldsymbol{w}_{t}(s)ds\right\Vert \left\Vert \nabla \boldsymbol{v}\right\Vert \nonumber \\
	& \leq C\left(\left\Vert e_{\boldsymbol{w}}^{n}\right\Vert ^{2}+\left\Vert \int_{t^{n}}^{t^{n+1}}\boldsymbol{w}_{t}(s)ds\right\Vert \right)\left\Vert \nabla \boldsymbol{v}\right\Vert.\label{eq:e_error_p0}
\end{align}
Invoking with \eqref{eq:e_estimate for trilinear form0}
and \eqref{eq:gradu}, we have
\begin{align}
 & \begin{aligned}\exp\end{aligned}
(\frac{t^{n+1}}{T})\left(q(t^{n+1})(\boldsymbol{u}(t^{n+1})\cdot\nabla)\boldsymbol{u}(t^{n+1})-q^{n+1}(\boldsymbol{u}^{n}\cdot\nabla)\boldsymbol{u}^{n},\boldsymbol{v}\right)\nonumber \\
 & =\left((\boldsymbol{u}(t^{n+1})-\boldsymbol{u}^{n})\cdot\nabla\boldsymbol{u}(t^{n+1}),\boldsymbol{v}\right)-e_{q}^{n+1}\exp{\left(\frac{t^{n+1}}{T}\right)}\left((\boldsymbol{u}^{n}\cdot\nabla)\boldsymbol{u}^{n},\boldsymbol{v}\right)+\left(\boldsymbol{u}^{n}\cdot\nabla(\boldsymbol{u}(t^{n+1})-\boldsymbol{u}^{n}),\boldsymbol{v}\right)\nonumber \\
 & \le C_{b,0}\left\Vert \nabla\boldsymbol{u}(t^{n+1})-\nabla\boldsymbol{u}^{n}\right\Vert \left\Vert\nabla \boldsymbol{u}(t^{n+1})\right\Vert\left\Vert \nabla\boldsymbol{v}\right\Vert +C_{b,0}\exp(1)|e_{q}^{n+1}|\left\Vert \nabla\boldsymbol{u}^{n}\right\Vert \left\Vert \nabla\boldsymbol{u}^{n}\right\Vert \left\Vert \nabla\boldsymbol{v}\right\Vert \nonumber \\
 & \quad+C_{b,0}\left\Vert \nabla\boldsymbol{u}^{n}\right\Vert \left\Vert \nabla\boldsymbol{u}(t^{n+1})-\nabla\boldsymbol{u}^{n}\right\Vert \left\Vert \nabla\boldsymbol{v}\right\Vert \nonumber \\
 & \leq C\left(\left\Vert \nabla e_{\boldsymbol{u}}^{n}\right\Vert +\left\Vert \int_{t^{n}}^{t^{n+1}}\nabla\boldsymbol{u}_{t}(s)ds\right\Vert +|e_{q}^{n+1}|\right)\left\Vert \nabla\boldsymbol{v}\right\Vert .\label{eq:e_error_p2}
\end{align}
Using Theorem \ref{thm: error_estimate_uwq}, Lemma \ref{lem:est_du}
and the inf-sup condition
\begin{equation}
\left\Vert e_{p}^{n+1}\right\Vert \leq\sup_{\boldsymbol{v}\in\boldsymbol{X}}\frac{(\nabla e_{p}^{n+1},\boldsymbol{v})}{\left\Vert \nabla\boldsymbol{v}\right\Vert },\label{eq:e_error_p3}
\end{equation}
we get
\begin{align*}
\tau\sum_{n=0}^{m}\left\Vert e_{p}^{n+1}\right\Vert ^{2} & \le C\tau\sum_{n=0}^{m}\left(\left\Vert \delta_{t}e_{\boldsymbol{u}}^{n+1}\right\Vert ^{2}+\nu_{0}\left\Vert \nabla e_{\boldsymbol{u}}^{n+1}\right\Vert ^{2}+\left\Vert \nabla e_{\boldsymbol{u}}^{n}\right\Vert ^{2}+|e_{q}^{n+1}|^{2}+\left\Vert e_{\boldsymbol{w}}^{n}\right\Vert ^{2}\right)\\
 & \quad+C\tau^{2}\int_{0}^{t^{m+1}}\left(\left\Vert \nabla\boldsymbol{u}_{t}(s)\right\Vert ^{2}+\left\Vert \boldsymbol{w}_{t}(s)\right\Vert ^{2}\right)ds\\
 & \le C\tau^{2}.
\end{align*}
The proof is complete. 
\end{proof}

\begin{rem}
In this paper, we only consider the first-order scheme for MNS equations. 	However, the SAV approach is known as a useful tool to design high-order schemes. What prevents us from studying the second-order scheme is that the SAV approach only used for the convective terms and the coupling terms are needed to deal with some subtle IMEX treatments. Undoubtedly, one can get a second-order scheme based on the second-order backward difference formula by using the SAV approach for the convective terms and making full-implicit treatments for the coupling terms. Although the obtained scheme is second-order and unconditionally energy stable, it is fully decoupled and thus one needs to solve large linear systems at each time step. Alternatively, one can also design a second-order scheme with a decoupled structure by replacing the full-implicit treatments with some subtle IMEX treatments for the coupling terms in \cite{Nochetto2014,Salgado2015}. By using the arguments in this paper, it is easy to give the stability and error estimates for the resulting scheme. But it is a pity that the scheme is only conditionally energy stable in the sense that the energy stability only holds with a time step restriction.  
\end{rem}

\section{Numerical experiments\label{sec:Num}}

In this section, we provide some numerical examples to verify
the theoretical findings of the proposed scheme. In all examples below, the spatial discretization is based
on mixed finite element method. To be more specific, we use the $P_{2}$ element for the 
angular velocity, the inf--sup stable $\boldsymbol{P}_{2}/P_{1}$ element
for the velocity and pressure. We always fix the mesh size with $h=1/150$ so that the spatial
discretization error is negligible compared to the time discretization error. The numerical experiments
are carried out using the finite element software FreeFem++ \citep{Hecht2012}.
Without further specified, we take $\jmath=1$, $c_{1}=2$ and $c_{2}=1$.

\begin{example}[Convergence test]
In this example, the computational domain is taken as $\Omega=(0,1)^{2}$
and the final time is chosen as $T=1$. Consider the following solution
to the MNS equations with external forces $\boldsymbol{f}$ and $g$
in the momentum equation and angular momentum equation,
\[
\boldsymbol{u}=(\sin t\sin^{2}(\pi x)\sin(2\pi y),-\sin t\sin(2\pi x)\sin^{2}(\pi y)),\, p=\sin t\sin(\pi x)\sin(\pi y),\, w=\sin t\sin^{2}(\pi x)\sin^{2}(\pi y).
\]
\end{example}
In Tables \ref{tab:Nu1}-\ref{tab:Nu1e-2}, we present the numerical
results for $\nu=\nu_{r}=1$, 0.01 and 0.001. From these tables, we
observe that the scheme achieve the expected convergence rates in
time, which are consistent with the error estimates in Theorems \ref{thm: error_estimate_uwq}
and \ref{thm:error_estimate_p}.

\begin{table}
\begin{centering}
\caption{Errors and convergence rates for the MNS equations with $\nu=\nu_{r}=1$.
\label{tab:Nu1}}
\par\end{centering}
\centering{}%
\begin{tabular}{|c|c|c|c|c|c|c|}
\hline 
$\tau$ & $\left\Vert e_{\boldsymbol{u}}^{N}\right\Vert$ & $\left\Vert \nabla e_{\boldsymbol{u}}^{N}\right\Vert$ & $\left\Vert e_{p}^{N}\right\Vert$ & $\left\Vert e_{w}^{N}\right\Vert$ & $\left\Vert \nabla e_{w}^{N}\right\Vert$ & $\left|e_{q}^{N}\right|$\\
\hline 
0.2 & 5.23e-3(---) & 3.78e-2(---) & 5.07e-2(---) & 1.60e-3(---) & 7.67e-3(---) & 3.40e-2(---)\\
\hline 
0.1 & 2.47e-3(1.08) & 1.79e-2(1.08) & 2.29e-2(1.15) & 7.86e-4(1.03) & 3.75e-3(1.03) & 1.77e-2(0.94)\\
\hline 
0.05 & 1.20e-3(1.04) & 8.71e-3(1.04) & 1.07e-2(1.09) & 3.89e-4(1.02) & 1.85e-3(1.01) & 9.01e-3(0.97)\\
\hline 
0.025 & 5.91e-4(1.02) & 4.31e-3(1.02) & 5.16e-3(1.05) & 1.93e-4(1.01) & 9.38e-4(0.99) & 4.55e-3(0.99)\\
\hline 
\end{tabular}
\end{table}

\begin{table}
\begin{centering}
\caption{Errors and convergence rates for the MNS equations with $\nu=\nu_{r}=$0.1.
\label{tab:Nu1e-1}}
\par\end{centering}
\centering{}%
\begin{tabular}{|c|c|c|c|c|c|c|}
\hline 
$\tau$ & $\left\Vert e_{\boldsymbol{u}}^{N}\right\Vert$ & $\left\Vert \nabla e_{\boldsymbol{u}}^{N}\right\Vert$ & $\left\Vert e_{p}^{N}\right\Vert$ & $\left\Vert e_{w}^{N}\right\Vert$ & $\left\Vert \nabla e_{w}^{N}\right\Vert$ & $\left|e_{q}^{N}\right|$\\
\hline 
0.2 & 9.34e-3(---) & 6.78e-2(---) & 4.98e-2(---) & 8.89e-4(---) & 4.03e-3(---) & 3.40e-2(---)\\
\hline 
0.1 & 4.64e-3(1.01) & 3.39e-2(1.01) & 2.27e-2(1.13) & 4.54e-4(0.97) & 2.07e-3(0.97) & 1.77e-2(0.94)\\
\hline 
0.05 & 2.31e-3(1.01) & 1.69e-2(1.01) & 1.07e-2(1.09) & 2.29e-4(0.99) & 1.05e-3(0.97) & 9.01e-3(0.97)\\
\hline 
0.025 & 1.15e-3(1.00) & 8.40e-3(1.00) & 5.18e-3(1.05) & 1.15e-4(0.99) & 5.52e-4(0.94) & 4.55e-3(0.99)\\
\hline 
\end{tabular}
\end{table}

\begin{table}
\begin{centering}
\caption{Errors and convergence rates for the MNS equations with $\nu=\nu_{r}=$0.01.
\label{tab:Nu1e-2}}
\par\end{centering}
\centering{}%
\begin{tabular}{|c|c|c|c|c|c|c|}
\hline 
$\tau$ & $\left\Vert e_{\boldsymbol{u}}^{N}\right\Vert$ & $\left\Vert \nabla e_{\boldsymbol{u}}^{N}\right\Vert$ & $\left\Vert e_{p}^{N}\right\Vert$ & $\left\Vert e_{w}^{N}\right\Vert$ & $\left\Vert \nabla e_{w}^{N}\right\Vert$ & $\left|e_{q}^{N}\right|$\\
\hline 
0.2 & 2.41e-2(---) & 2.06e-1(---) & 5.54e-2(---) & 7.59e-4(---) & 3.42e-3(---) & 3.40e-2(---)\\
\hline 
0.1 & 1.21e-2(1.00) & 1.01e-2(1.03) & 2.69e-2(1.04) & 3.90e-4(0.96) & 1.77e-3(0.96) & 1.77e-2(0.94)\\
\hline 
0.05 & 6.05e-3(1.00) & 4.97e-2(1.02) & 1.31e-2(1.03) & 1.97e-4(0.98) & 9.07e-4(0.97) & 9.01e-3(0.97)\\
\hline 
0.025 & 3.03e-3(1.00) & 2.47e-2(1.01) & 6.51e-3(1.02) & 9.92e-4(0.99) & 4.83e-4(0.91) & 4.55e-3(0.99)\\
\hline 
\end{tabular}
\end{table}

\begin{example}[Stability test]
 This example is to test the energy stability of the proposed scheme.
For this end, we take the domain as $\Omega=(0,1)^{2}$ and set the
initial conditions for $\boldsymbol{u}$ and $w$ to be 
\[
\boldsymbol{u}^{0}=(x^{2}(x-1)^{2}y(y-1)(2y-1),-y^{2}(y-1)^{2}x(x-1)(2x-1)),\quad w^{0}=\sin(\pi x)\sin(\pi y).
\]
\end{example}
We carry out the numerical experiments with different physical parameters
and different time-steps. Figure \ref{fig:Energy} presents the time
evolutions of the discrete energy with the finial time $T=5$. We
observe that all energy curves decay monotonically, which numerically
confirms that our scheme is unconditionally energy stable.

\begin{figure}
\hfill{}\subfloat[$\nu=\nu_{r}=0.1$]{\begin{centering}
\includegraphics[scale=0.5]{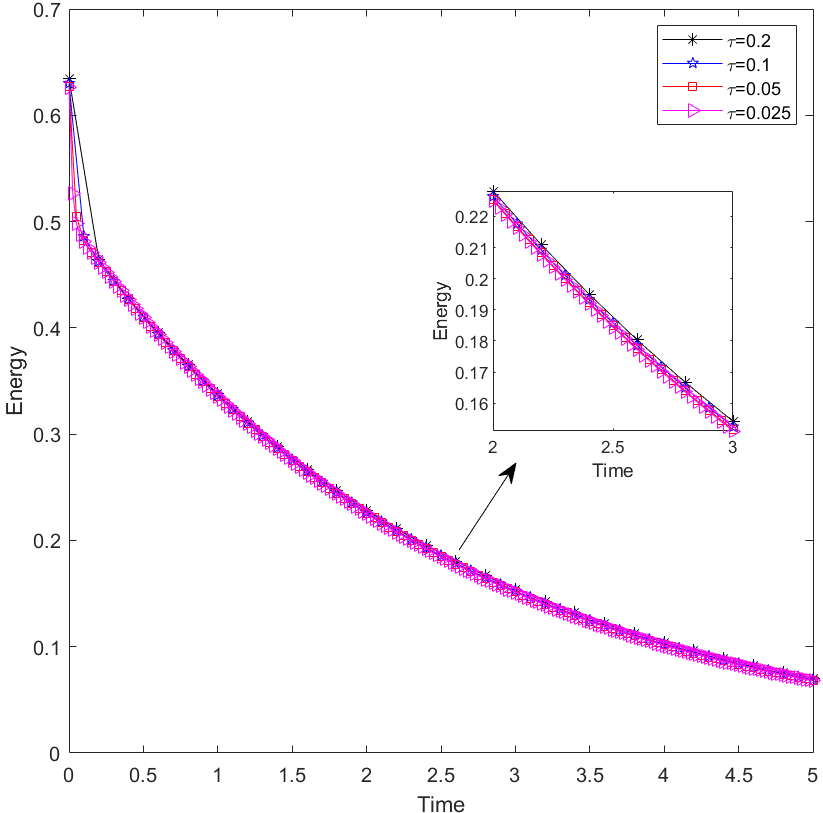}
\par\end{centering}
}\hfill{}\subfloat[$\nu=\nu_{r}=0.01$]{\begin{centering}
\includegraphics[scale=0.5]{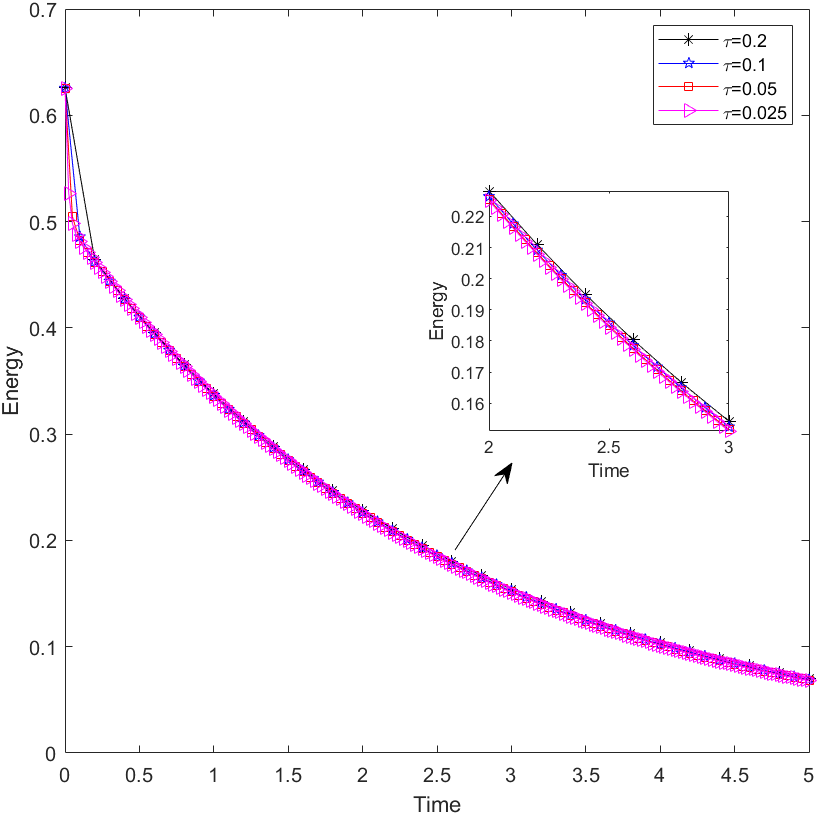}
\par\end{centering}
}\hfill{}

\caption{Time evolution of the energy with different time step sizes.\label{fig:Energy}}
\end{figure}

\begin{example}[Stirring of a Passive Scalar]
This example is to compute a realistic example about the stirring
of a passive scalar. To simulate this example, we supplement the
MNS equations \eqref{eq:MNS_model} with the following convection
equation:
\begin{equation}
\phi_{t}+\boldsymbol{u}\cdot\nabla\phi=0,\label{eq:convect}
\end{equation}
where $\boldsymbol{u}$ is the velocity from the MNS equations \eqref{eq:MNS_model},
$\phi$ denotes the passive scalar whose value does not affect the
flow. Since there is no diffusion in \eqref{eq:convect}, thus mixing
depends only on the flow pattern. The computational domain is set
by $\Omega=(-1,1)^{2}$, the time-step is chosen as $\tau=0.01$
and the final time is taken as $T=25$. The angular momentum equation
is supplemented by an external force $g=25\left(x-1\right)$, and
the initial conditions are described by $\boldsymbol{u}^{0}=\boldsymbol{0}$,
$w^{0}=0$ and 
\[
\phi(x,0)=\begin{cases}
1, & x_{2}<0,\\
0, & x_{2}\geq0.
\end{cases}
\]
The profile of the initial scalar is shown in Figure \ref{fig:Init}.
It is remarked that the MNS equations and the convection equation
are not solved simultaneously in our implementation. At every time-step,
we first solve the MNS equations using the proposed schemes, and then
we solve the convection equation \eqref{eq:convect} by using characteristics-Galerkin
method with $P_{2}$ finite element. 
\end{example}

\begin{figure}
\begin{centering}
\includegraphics[scale=0.4]{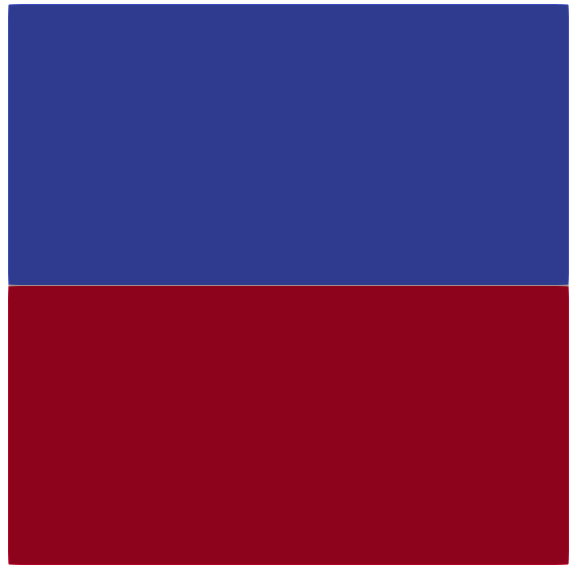}
\par\end{centering}
\caption{Initial profile of the scalar $\phi^{0}(x)$.\label{fig:Init}}
\end{figure}

The evolutions of $\phi$ with different values of $\nu=\nu_{r}$
are shown in Figures \ref{fig:Nu1}-\ref{fig:Nu1e-2}. Since a linear
velocity is generated by the applied torque, the scalar begins
to convect by the flow and we can observe the evolution of the variable
$\phi$. We can also see that the mixing of $\nu=\nu_{r}=0.1$ is
most fast one, and $\nu=\nu_{r}=0.001$ is most slow one. The obtained
results coincide well with those discussed in \citep{Nochetto2014,Salgado2015}.

\begin{figure}
\hfill{}\subfloat[$t=1$]{\begin{centering}
\includegraphics[scale=0.3]{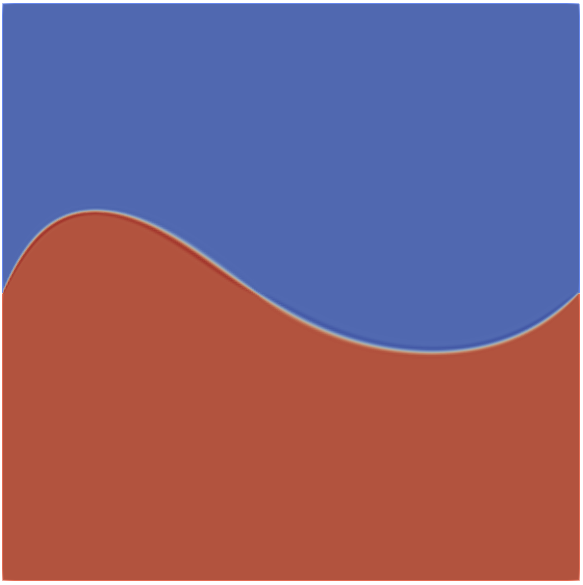}
\par\end{centering}

}\hfill{}\subfloat[$t=5$]{\begin{centering}
\includegraphics[scale=0.3]{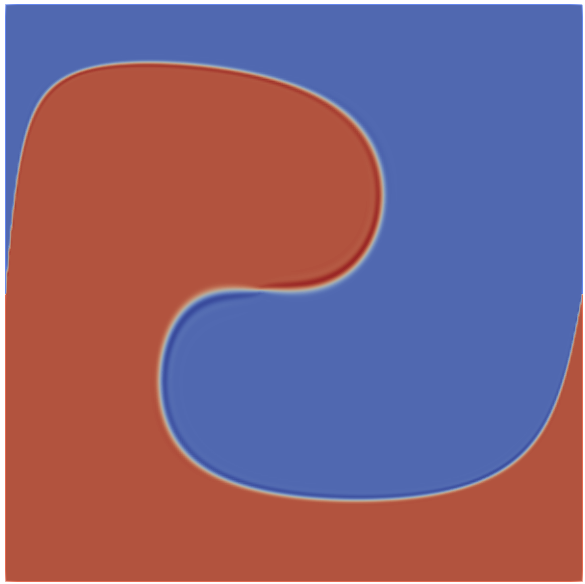}
\par\end{centering}

}\hfill{}\subfloat[$t=7$]{\begin{centering}
\includegraphics[scale=0.3]{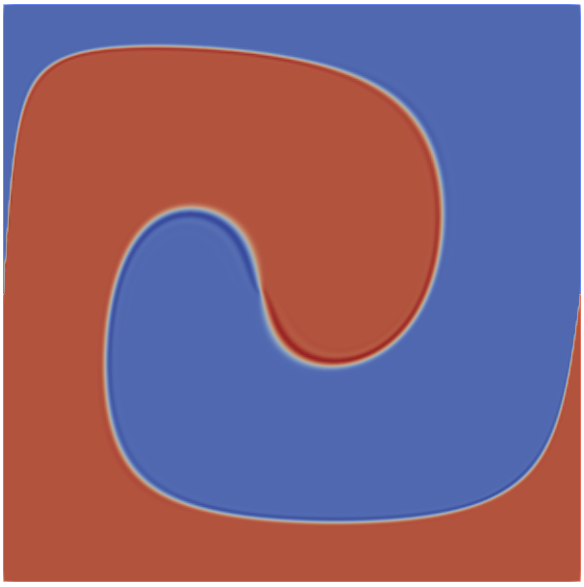}
\par\end{centering}

}\hfill{}

\hfill{}\subfloat[$t=10$]{\begin{centering}
\includegraphics[scale=0.3]{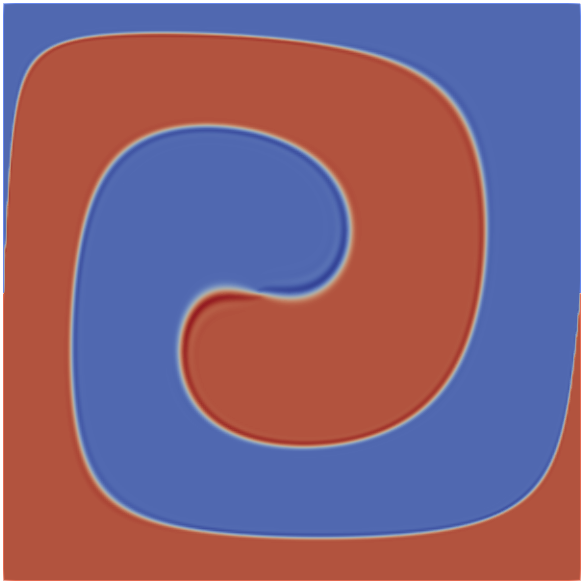}
\par\end{centering}
}\hfill{}\subfloat[$t=15$]{\begin{centering}
\includegraphics[scale=0.3]{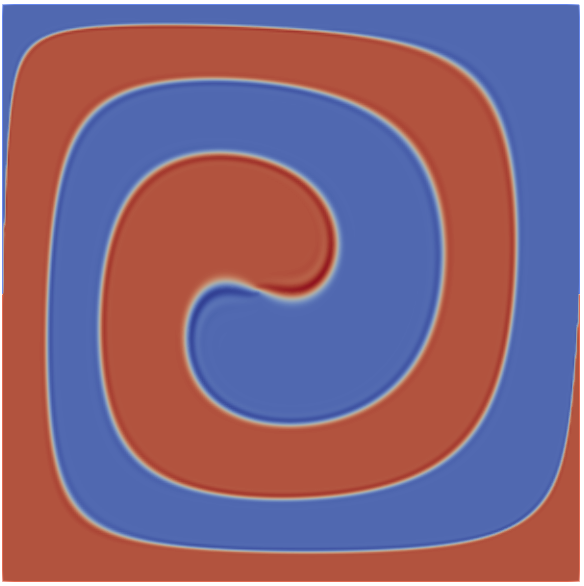}
\par\end{centering}
}\hfill{}\subfloat[$t=18$]{\begin{centering}
\includegraphics[scale=0.3]{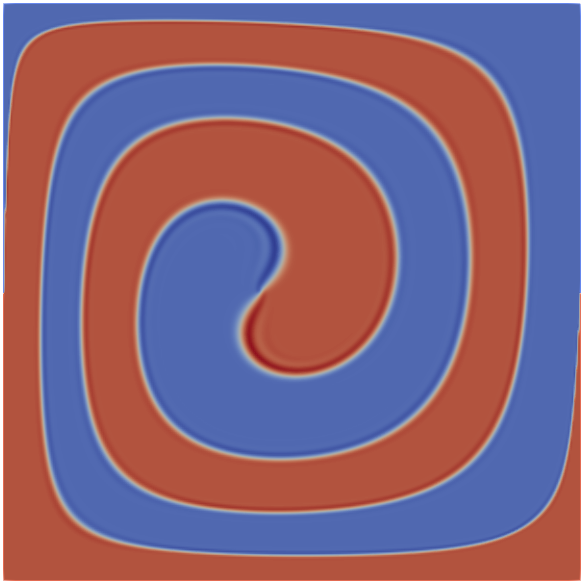}
\par\end{centering}
}\hfill{}

\hfill{}\subfloat[$t=20$]{\begin{centering}
\includegraphics[scale=0.3]{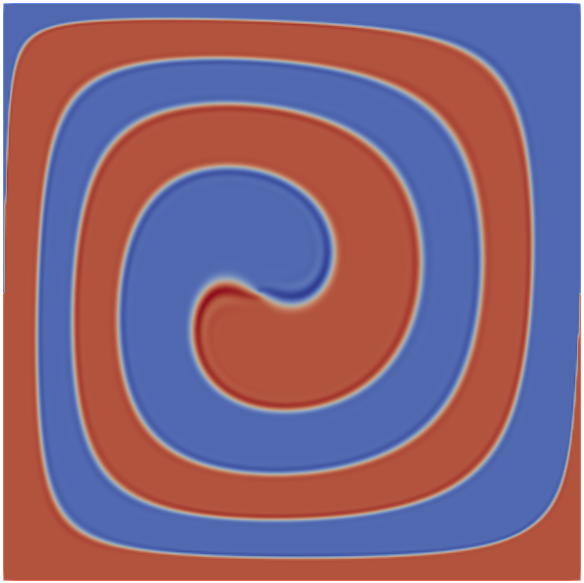}
\par\end{centering}
}\hfill{}\subfloat[$t=23$]{\begin{centering}
\includegraphics[scale=0.3]{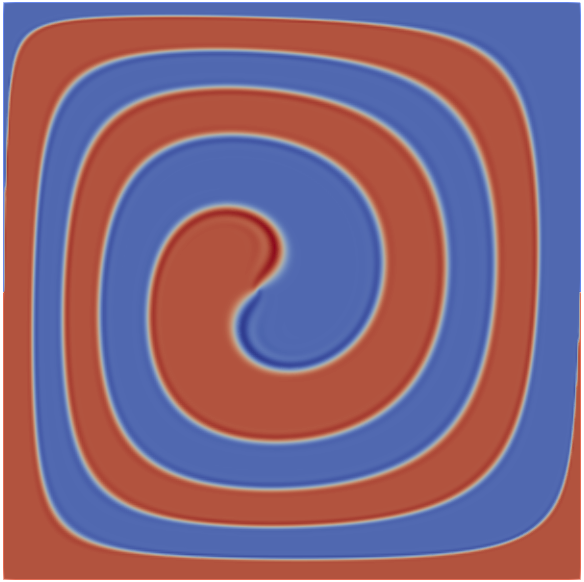}
\par\end{centering}
}\hfill{}\subfloat[$t=25$]{\begin{centering}
\includegraphics[scale=0.3]{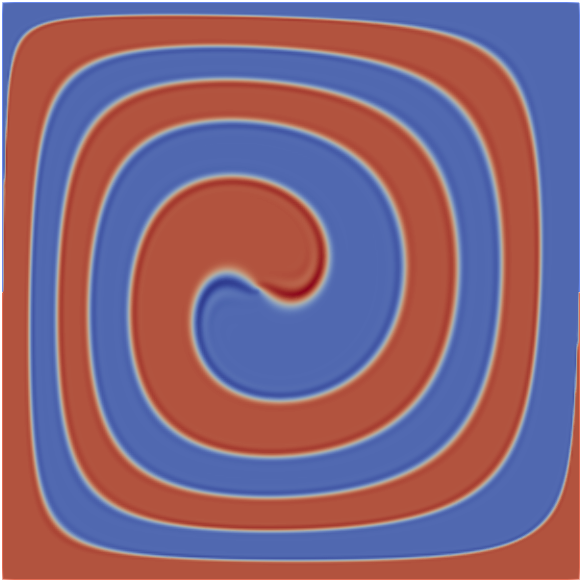}
\par\end{centering}
}\hfill{}

\caption{Mixing of a convected passive scalar $\phi$ by means of an applied
torque with $\nu=\nu_{r}=1$.\label{fig:Nu1}}

\end{figure}

\begin{figure}
\hfill{}\subfloat[$t=1$]{\begin{centering}
\includegraphics[scale=0.3]{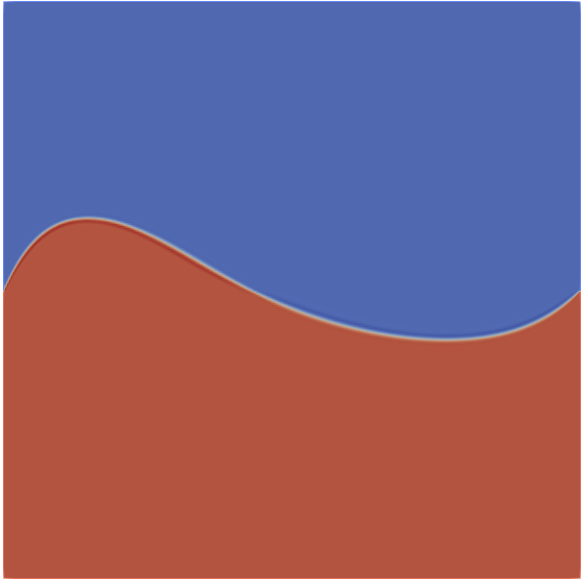}
\par\end{centering}
}\hfill{}\subfloat[$t=5$]{\begin{centering}
\includegraphics[scale=0.3]{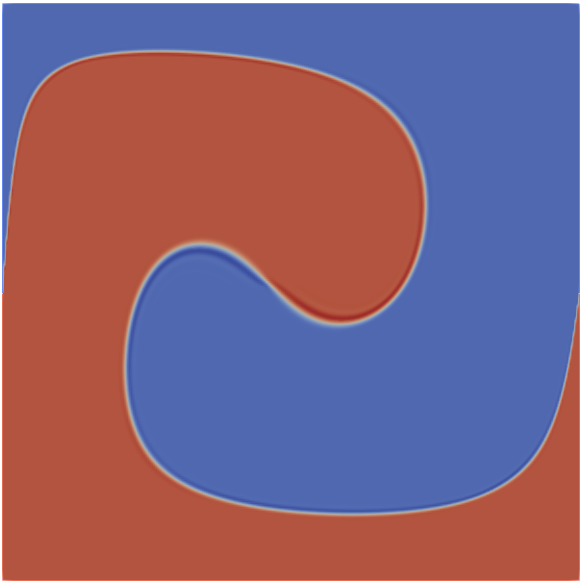}
\par\end{centering}
}\hfill{}\subfloat[$t=7$]{\begin{centering}
\includegraphics[scale=0.3]{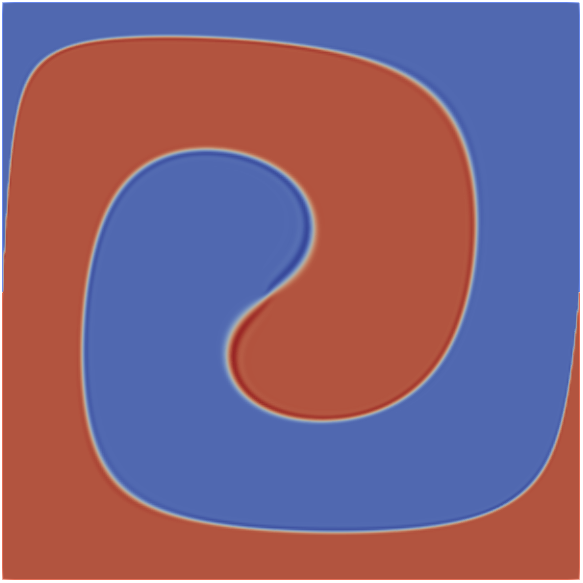}
\par\end{centering}
}\hfill{}

\hfill{}\subfloat[$t=10$]{\begin{centering}
\includegraphics[scale=0.3]{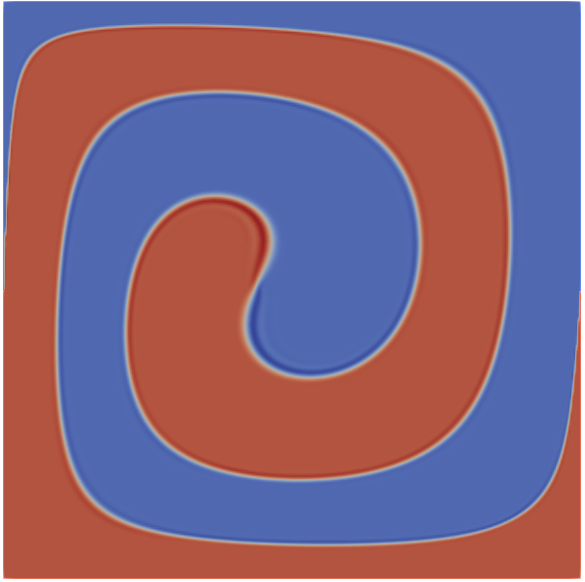}
\par\end{centering}
}\hfill{}\subfloat[$t=15$]{\begin{centering}
\includegraphics[scale=0.3]{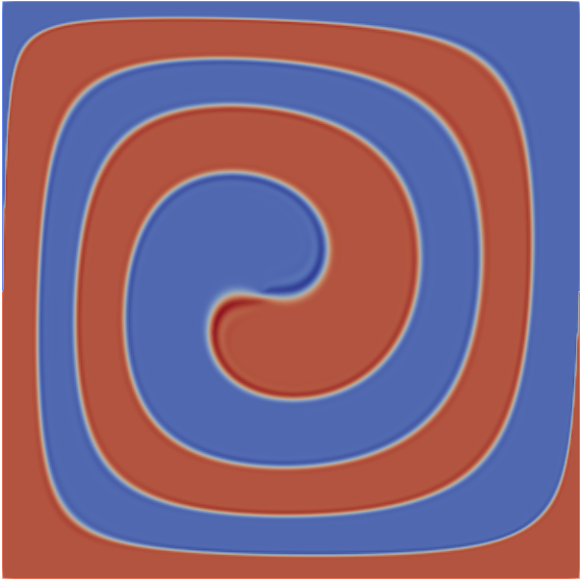}
\par\end{centering}
}\hfill{}\subfloat[$t=18$]{\begin{centering}
\includegraphics[scale=0.3]{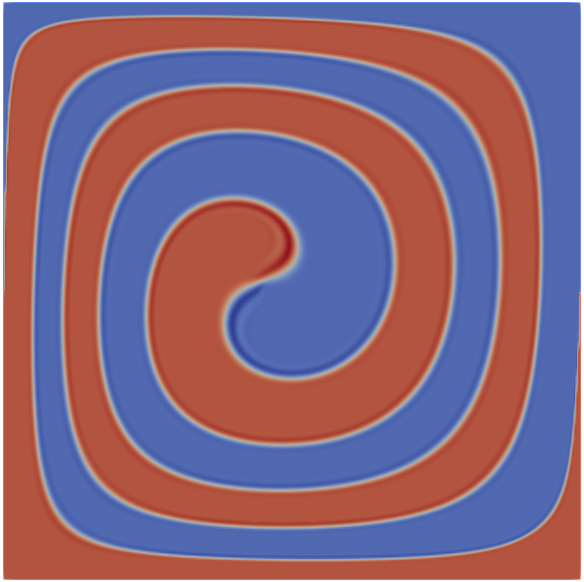}
\par\end{centering}
}\hfill{}

\hfill{}\subfloat[$t=20$]{\begin{centering}
\includegraphics[scale=0.3]{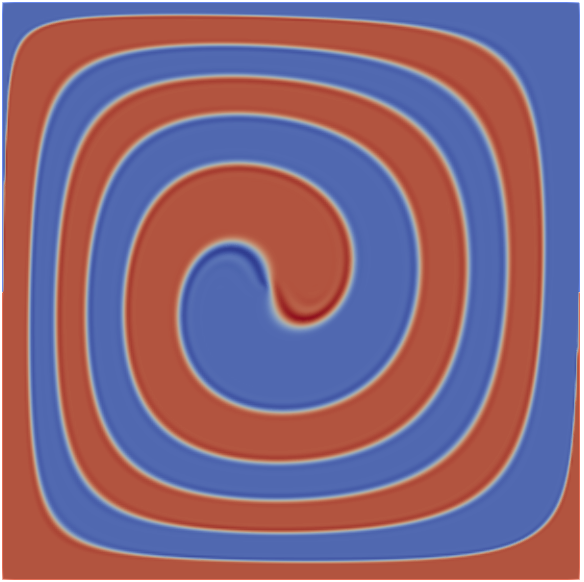}
\par\end{centering}
}\hfill{}\subfloat[$t=23$]{\begin{centering}
\includegraphics[scale=0.3]{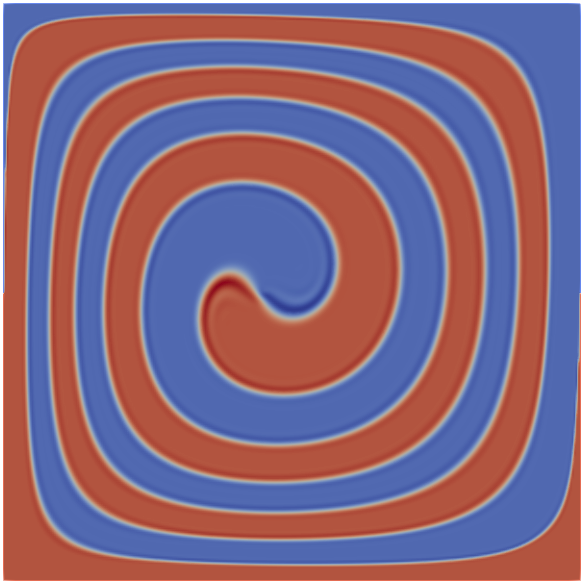}
\par\end{centering}
}\hfill{}\subfloat[$t=25$]{\begin{centering}
\includegraphics[scale=0.3]{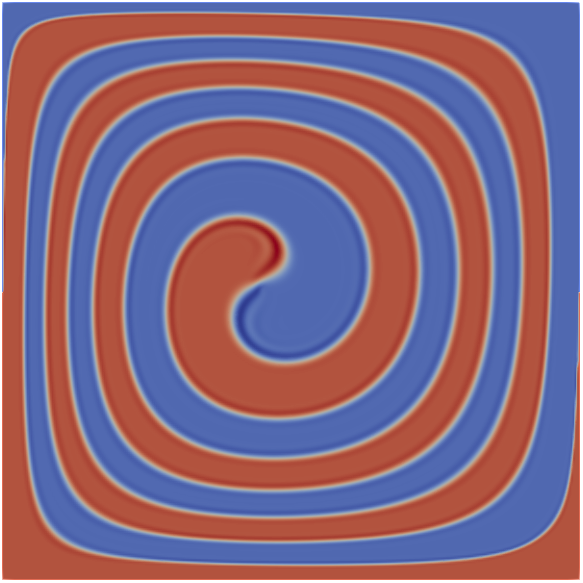}
\par\end{centering}
}\hfill{}

\caption{Mixing of a convected passive scalar $\phi$ by means of an applied
torque with $\nu=\nu_{r}=0.1$.\label{fig:Nu1e-1}}
\end{figure}

\begin{figure}
\hfill{}\subfloat[$t=1$]{\begin{centering}
\includegraphics[scale=0.3]{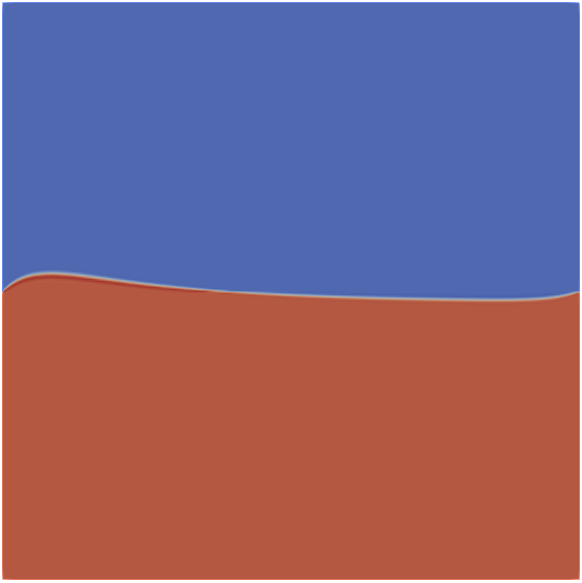}
\par\end{centering}
}\hfill{}\subfloat[$t=5$]{\begin{centering}
\includegraphics[scale=0.3]{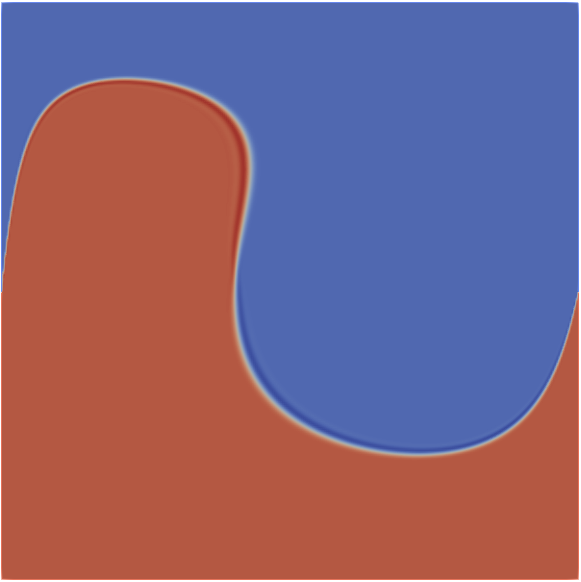}
\par\end{centering}
}\hfill{}\subfloat[$t=7$]{\begin{centering}
\includegraphics[scale=0.3]{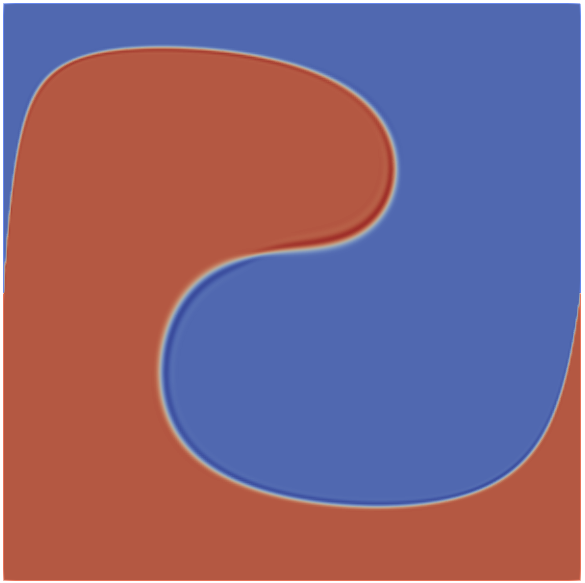}
\par\end{centering}
}\hfill{}

\hfill{}\subfloat[$t=10$]{\begin{centering}
\includegraphics[scale=0.3]{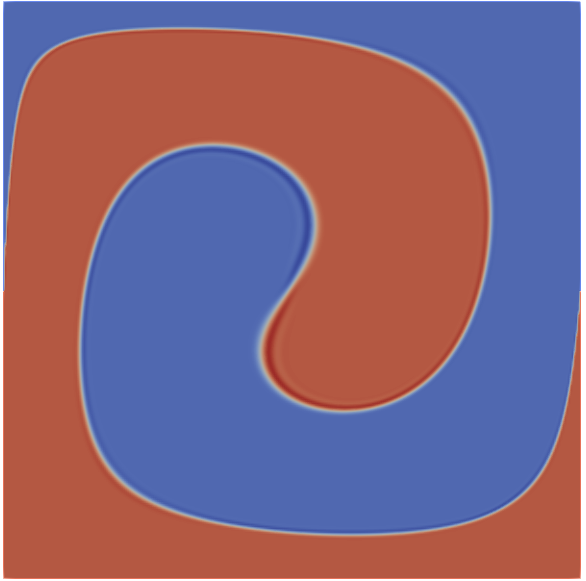}
\par\end{centering}
}\hfill{}\subfloat[$t=15$]{\begin{centering}
\includegraphics[scale=0.3]{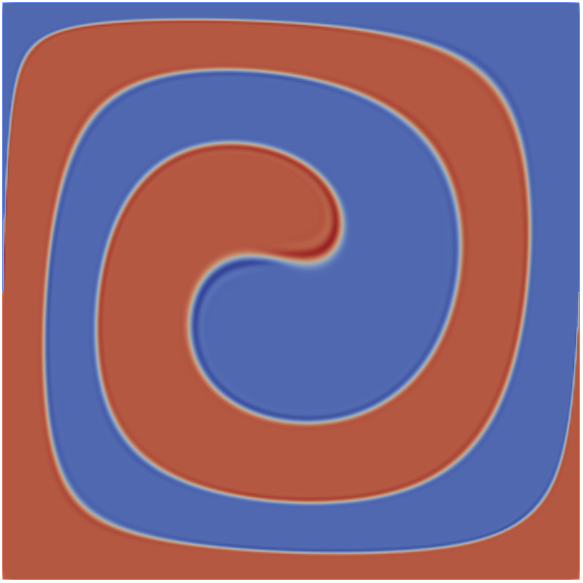}
\par\end{centering}
}\hfill{}\subfloat[$t=18$]{\begin{centering}
\includegraphics[scale=0.3]{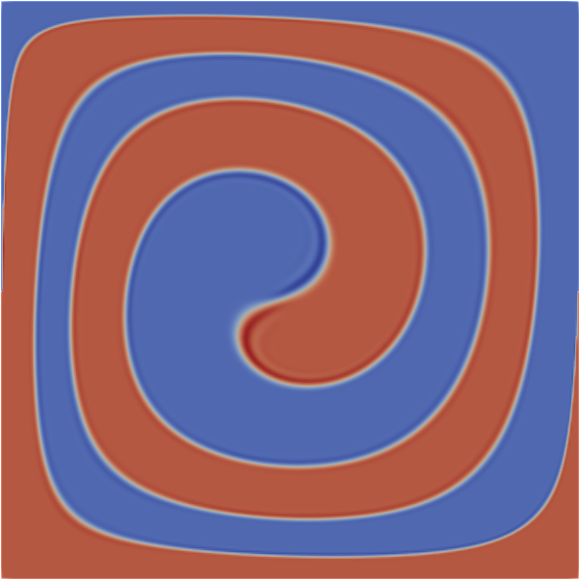}
\par\end{centering}
}\hfill{}

\hfill{}\subfloat[$t=20$]{\begin{centering}
\includegraphics[scale=0.3]{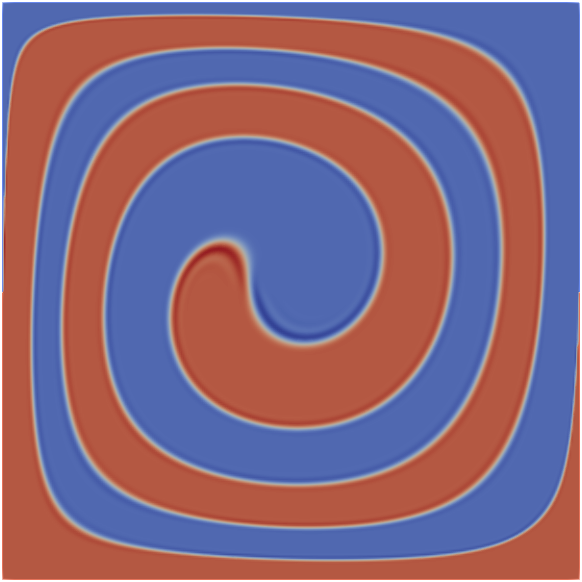}
\par\end{centering}
}\hfill{}\subfloat[$t=23$]{\begin{centering}
\includegraphics[scale=0.3]{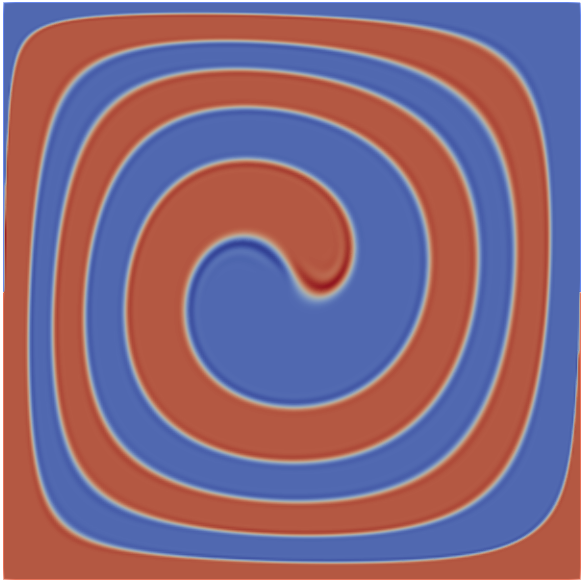}
\par\end{centering}
}\hfill{}\subfloat[$t=25$]{\begin{centering}
\includegraphics[scale=0.3]{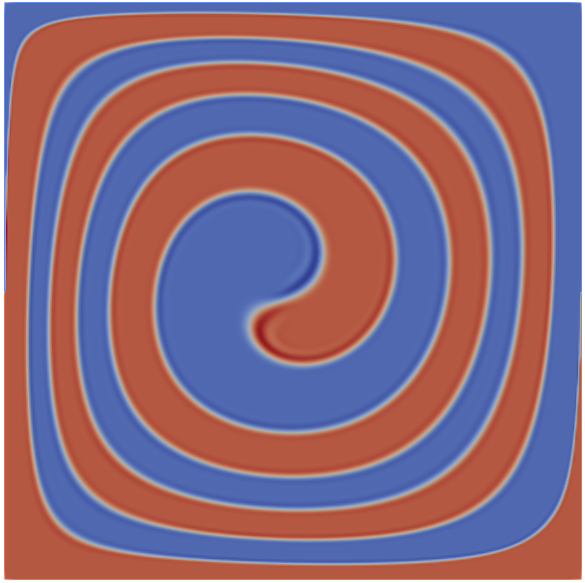}
\par\end{centering}
}\hfill{}

\caption{Mixing of a convected passive scalar $\phi$ by means of an applied
torque with $\nu=\nu_{r}=0.01$.\label{fig:Nu1e-2}}
\end{figure}

\begin{figure}
\hfill{}\subfloat[$t=1$]{\begin{centering}
\includegraphics[scale=0.3]{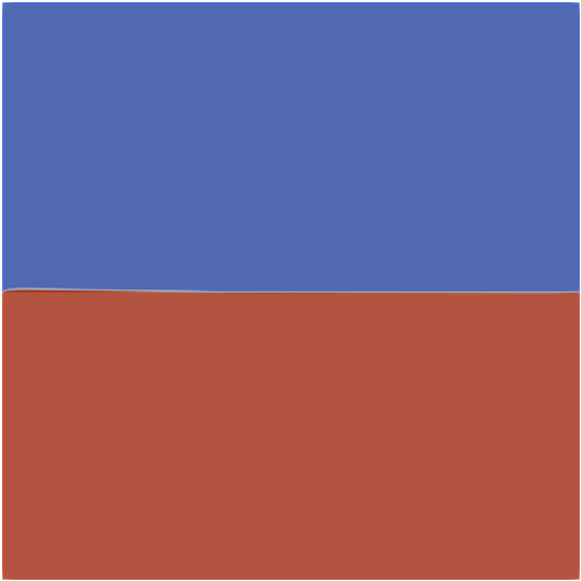}
\par\end{centering}
}\hfill{}\subfloat[$t=5$]{\begin{centering}
\includegraphics[scale=0.3]{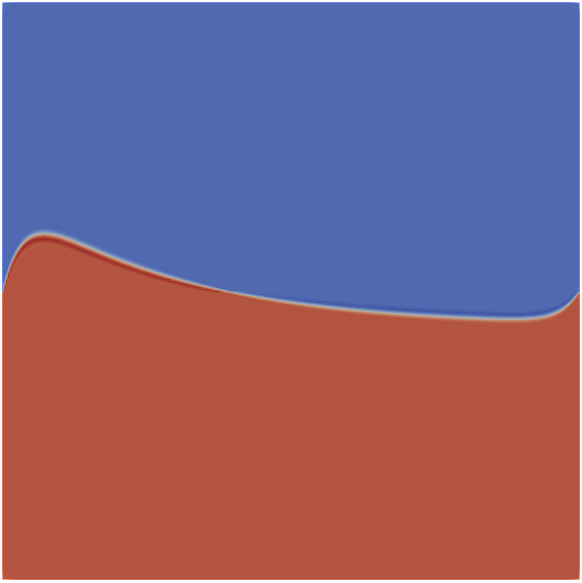}
\par\end{centering}
}\hfill{}\subfloat[$t=7$]{\begin{centering}
\includegraphics[scale=0.3]{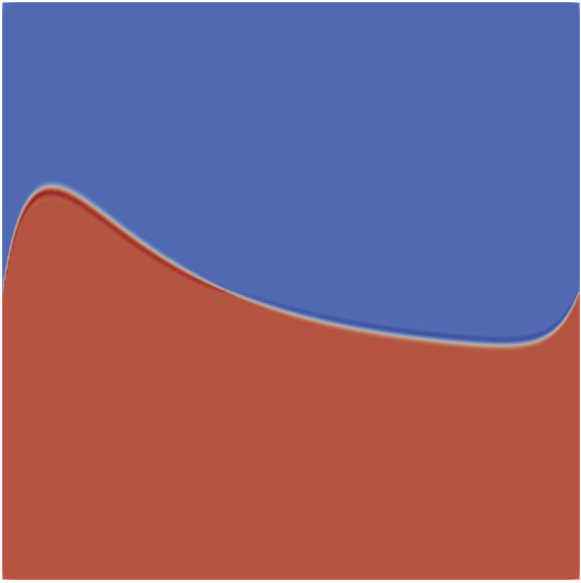}
\par\end{centering}
}\hfill{}

\hfill{}\subfloat[$t=10$]{\begin{centering}
\includegraphics[scale=0.3]{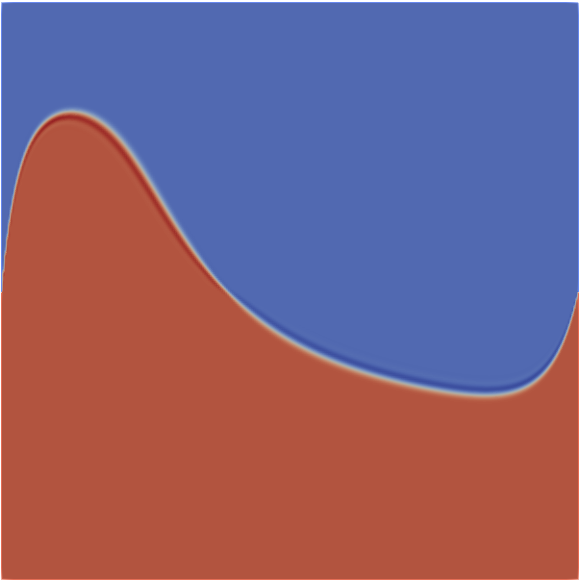}
\par\end{centering}
}\hfill{}\subfloat[$t=15$]{\begin{centering}
\includegraphics[scale=0.3]{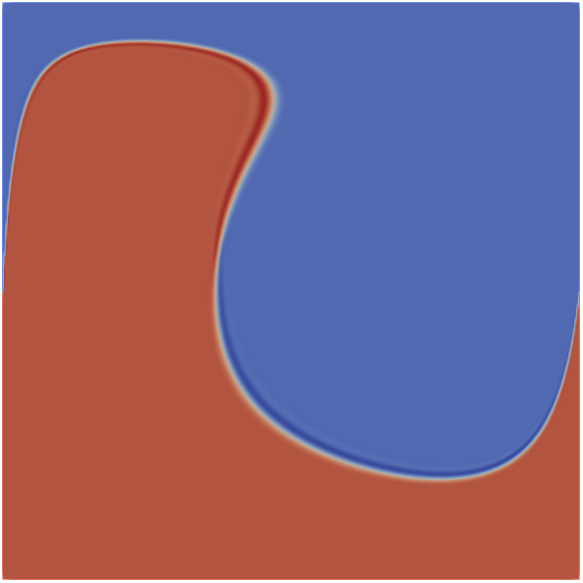}
\par\end{centering}
}\hfill{}\subfloat[$t=18$]{\begin{centering}
\includegraphics[scale=0.3]{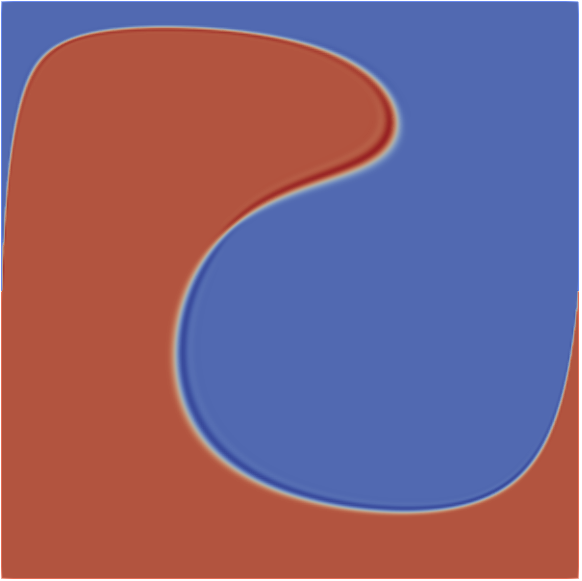}
\par\end{centering}
}\hfill{}

\hfill{}\subfloat[$t=20$]{\begin{centering}
\includegraphics[scale=0.3]{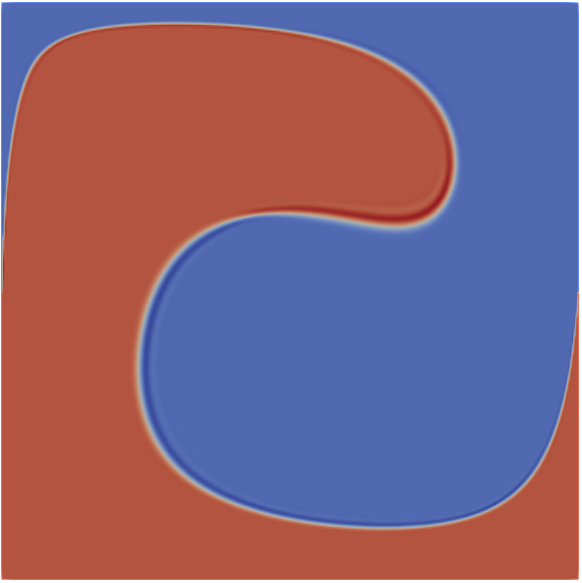}
\par\end{centering}
}\hfill{}\subfloat[$t=23$]{\begin{centering}
\includegraphics[scale=0.3]{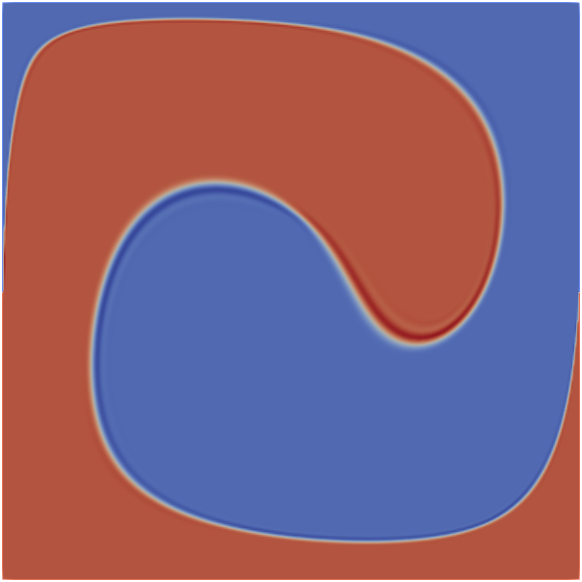}
\par\end{centering}
}\hfill{}\subfloat[$t=25$]{\begin{centering}
\includegraphics[scale=0.3]{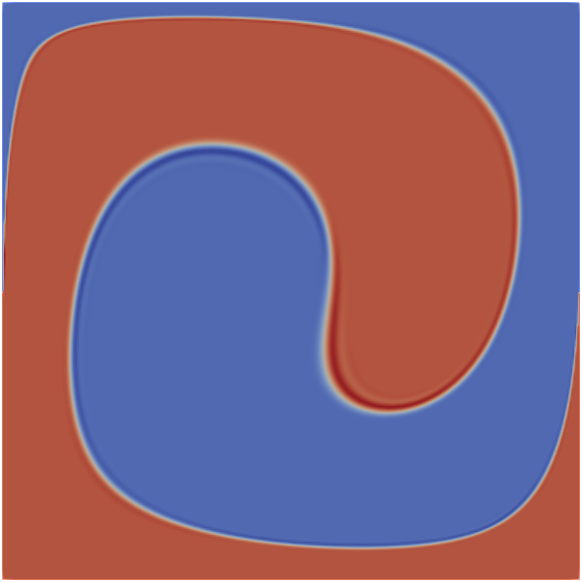}
\par\end{centering}
}\hfill{}

\caption{Mixing of a convected passive scalar $\phi$ by means of an applied
torque with $\nu=\nu_{r}=0.001$.\label{fig:Nu1e-3}}
\end{figure}

\section{Concluding remarks\label{sec:Conclud}}

In this paper, we propose and analyze a first-order discretization
scheme in time for the MNS equations. The scheme is based on the SAV
approach for the convective terms and some subtle implicit-explicit
treatments for the coupling terms. The attractive points of this scheme
are it is decoupled, linear, unconditionally energy stable and easy
to implement. We further derive rigorous error estimates in the two-dimensional
case without any condition on the time step. Some numerical
experiments are given to confirm the theoretical findings and show
the performances of the scheme. In the further, the error estimates
in three dimensions and the high order SAV schemes will be considered. 

\section*{References}

\bibliographystyle{unsrt}
\bibliography{Ref}

\end{document}